\documentclass[11pt]{amsart}

\newtheorem{theorem}{Theorem}[section]
\newtheorem{lemma}[theorem]{Lemma}
\newtheorem{proposition}[theorem]{Proposition}
\newtheorem{corollary}[theorem]{Corollary}
\newtheorem{algorithm}[theorem]{Algorithm}

\newtheorem{remark}{Remark}[section]

\def\bx{\mathbf x}
\def\rat#1{\mathcal R_{#1}}

\def\Forall{\quad \hbox{ for all }}
\def\bal#1\eal{\begin{aligned} #1 \end{aligned}}
\def\beq#1\eeq{\begin{equation} #1 \end{equation}}
\def\cI{\widetilde {\mathcal I}}
\usepackage[colorlinks,bookmarksopen,bookmarksnumbered,citecolor=red,urlcolor=red]{hyperref}
\usepackage{amsmath, amssymb, amsfonts, amsbsy, latexsym}
\usepackage{accents}
\usepackage{hhline}
\usepackage{multirow}
\usepackage{epsfig}
\usepackage{graphicx}
\def\argmin{\mathop{\rm argmin}}

\def\calA{{\mathcal A}}
\def\calAt{{\mathbb A}}

\def\calIt{{\mathbb I}}

\def\wcalSt{ \widetilde{\mathbb S}}
\def\wcalMt{ \widetilde{\mathbb M}}
\def\wcalAt{ \widetilde {\mathbb A}}
\def\wcalIt{\widetilde{\mathbb I}}

\def\calGt{ {\mathbb G}}

\newcommand{\R}{\mathbb{R}}

\newcommand{\RR}{\mathbb{R}}

\newcommand{\NN}{\mathbb{N}}

\newcommand{\calI}{{\cal I}}

\def\argmin{\mathop{\rm argmin}}

\def\tiluh{{\widetilde u}_h}
\def\tilfh{{\widetilde f}_h}
\def\tilwh{{\widetilde w}_h}
\def\tilvh{{\widetilde v}_h}
\def\tilF{{\widetilde F}}

\def\calSt{{\mathbb S}}
\def\calSt{\widetilde{\mathbb S}}
\def\calIt{{\mathbb I}}

\def\calI{\mathcal I}

\title[Analysis of BURA numerical methods for spectral fractional elliptic equations]
{Analysis of numerical methods for spectral fractional elliptic equations 
based on the best uniform rational approximation}
\author[S. Harizanov, R. Lazarov, P. Marinov, S. Margenov, J. Pasciak]
{Stanislav Harizanov \and Raytcho Lazarov \and Pencho Marinov \and Svetozar Margenov \and Joseph Pasciak}
\address{Institute of Information and Communication Technologies, Bulgarian Academy of 
Sciences, Acad. G. Bonchev, bl. 25A, 1113 Sofia, Bulgaria
(sharizanov@parallel.bas.bg)}
\address{Deptartment of Mathematics, Texas A\&M University, 
College Station, TX 77843-3368, USA (lazarov@math.tamu.edu) and Institute of Mathematics and Informatics,
Bulgarian Academy of Sciences, Acad. G. Bonchev, bl. 8, 1113 Sofia, Bulgaria}
\address{Institute of Information and Communication Technologies, Bulgarian Academy of 
Sciences, Acad. G. Bonchev, bl. 25A, 1113 Sofia, Bulgaria (pencho@parallel.bas.bg)}
\address{Institute of Information and Communication Technologies, Bulgarian Academy of 
Sciences, Acad. G. Bonchev, bl. 25A, 1113 Sofia, Bulgaria (margenov@parallel.bas.bg)}
\address{Deptartment of Mathematics, Texas A\&M University, College Station, TX 77843, USA (pasciak@math.tamu.edu)}

\begin{document}
\date{\today}

\begin{abstract}

Here we  study theoretically and compare experimentally with the methods developed  in \cite{HLMMV18,BP15}
an efficient method for solving systems of algebraic equations
$\wcalAt^\alpha \tiluh= \tilfh$, $0< \alpha <1$, where 
$\wcalAt$ is an $N \times N$ matrix coming from the discretization of a
fractional diffusion operator.
More specifically, we focus on matrices obtained from finite 
difference or finite element approximation of second 
order elliptic problems in $\R^d$, $d=1,2,3$. 
The proposed methods are based on the best uniform rational approximation (BURA)  
$r_{\alpha,k}(t)$ of $t^{\alpha}$  on $[0,1]$. Here $r_{\alpha,k} $ is a
rational function of $t$ involving numerator and denominator polynomials
of degree at most $k$.

The approximation of $ \tiluh= \wcalAt^{-\alpha} \tilfh$
is then $\tilwh=  \lambda_1^{-\alpha} r_{\alpha,k} (\lambda_1 \wcalAt^{-1}) \tilfh$,
where $\lambda_1$ is the smallest eigenvalue of $\wcalAt$.
We show that the proposed method is exponentially convergent with respect to $k$ and has some  
attractive properties. First, it reduces the solution of the nonlocal system to solution of $k$
systems with matrix $(\wcalAt +c_j \wcalIt)$ and $c_j>0$, $j=1,2,\ldots,k$. Thus, 
good computational complexity can be achieved if  fast solvers are available for such systems. 
Second, the original problem and its rational approximation in the
finite difference case are positivity preserving.  In the finite
element case, this valid for schemes obtained by mass lumping 
under certain mild conditions on the mesh. Further, we prove that the  lumped 
mass schemes still have the expected rate of convergence, at times assuming 
additional regularity on the right hand side. Finally, we present comprehensive 
numerical experiments on a number of  model problems for various $\alpha$ 
in one and two spatial dimensions.  These illustrate  the computational behavior of 
the proposed  method and compare its accuracy and efficiency with that of other methods
developed by Harizanov et. al. \cite{HLMMV18} and Bonito and Pasciak \cite{BP15} . 
\\[1ex]
\\[1ex]
\\[1ex]

\end{abstract}

\maketitle

\section{Introduction}\label{section1}
\subsection{Spectral fractional powers of elliptic operators} 
\label{ss:problem}
In this paper we consider the following second order elliptic equation with homogeneous Dirichlet data:
\beq
\bal
 - \nabla \cdot( a(x) \nabla v(x)) &= f(x),  &  \hbox{ for } x  \in \Omega,  \\
v(x)&=0,  &   \hbox{ for } x   \in \partial \Omega. 
\eal
\label{strong}
\eeq
Here $\Omega $ is a bounded domain in $\RR^d$,  $d\ge 1$,  and 
we assume that $0<a_0 \le a(x) $ for $x\in \Omega$.

The fractional powers of the elliptic operator associated with the problem \eqref{strong} are defined in 
terms of the weak form of
\eqref{strong}, namely, $v(x)$ is the unique function in  $V= H^1_0(\Omega)$
satisfying
\beq
a(v,\theta)= (f,\theta)\qquad \Forall \theta\in V.
\label{weak}
\eeq
Here
$$
a(w,\theta):=\int_\Omega  a(x) \nabla w(x) \cdot \nabla  \theta(x)\, dx  \quad \hbox{ and } \quad
(w,\theta):=\int_\Omega w(x) \theta(x)\, dx.
$$
For $f\in X := L^2(\Omega)$, \eqref{weak} defines a solution operator $T f :=v$.
Following \cite{kato}, we define an unbounded operator $\calA$ on $X$  as follows.
The operator $\calA$  with domain
 $$ D(\calA) = \{ Tf\,:\, f\in X\}$$
 is defined by  $\calA v=g$ for $v\in D(\calA)$ where  $g\in X$ with
 $Tg=v$.  This is well defined as $T$ is injective.

Thus, the focus of our work in this paper is numerical approximation and algorithm development for the
equation: 
\beq\label{frac-eq}
 \calA^\alpha u = f \quad  \mbox{with a solution} \quad u =\calA^{-\alpha} f.
\eeq
Here  $\calA^{-\alpha}=T^\alpha$ for $\alpha>0$
is defined by Dunford-Taylor integrals which can be transformed  when
$\alpha\in (0,1)$, to the  Balakrishnan integral, e.g. \cite{balakrishnan}: for $f\in X$,
\beq
u=\calA^{-\alpha} f =\frac {\sin(\pi \alpha)} \pi 
\int_0^\infty \mu^{-\alpha} (\mu \calI +\calA)^{-1}f\, d\mu.
\label{bal}
\eeq
This definition is sometimes referred to as the spectral definition
of fractional powers.   One can also use an equivalent definition through the
expansion with respect to the eigenfunctions of $\calA$, 
e.g. \cite{Karniadakis2018fractional,Acosta-Borthagaray2017}.
We note that there are also problems on bounded domains involving
fractional powers, for example, those related to L\'evy diffusion 
\cite{Acosta-Borthagaray2017,ros-oton2015}.   These problems
involve the restriction of non-local operators defined on $\RR^d$
applied to bounded domain functions extended, e.g., by 0, outside of $\Omega$.  
However, in this paper, we focus on the spectral definition \eqref{bal} and the corresponding
approximations by the finite element or finite difference methods.

An operator $L$ is positivity preserving if $Lf\ge 0$ when $f\ge 0$.
We note that by the maximum principle, $(\mu \calI+ \calA)^{-1}$ is a positivity
preserving operator for $\mu\ge 0$
and the formula \eqref{bal} shows that $\calA^{-\alpha}$ is  also.  In many
applications, it is important that the discrete approximations share
this property.

\subsection{Some semi-discrete schemes}\label{ss:semi-discrete}

We study approximations to $u= \calA^{-\alpha} f$ defined in terms  of finite difference or finite element
approximation of the operator  $T$.   We shall use the following convention regarding 
the approximate solutions by these two methods.   
The finite element solution is a function  in $V_h$,
an $N$-dimensional space of continuous piece-wise linear functions over a partition  ${\mathcal T}_h$ of
the domain. Such functions will be denoted by $u_h$, $ v_h$, etc.   
Also we shall denote by $\calAt$, $\calIt$, etc operators acting on the elements $u_h, \theta_h$, etc 
in the finite dimensional space of functions $V_h$.
When a nodal basis of the finite 
element space is introduced, then the vector coefficients in this basis
are denoted  $\tiluh$, $\tilvh$, etc. Under this convention  operator equations in $V_h$ such as 
$\calAt u_h=f_h$ will be written as a system of linear algebraic equations $ \wcalAt  \tiluh=\tilfh$  in $\RR^N$.

In the finite difference case, discrete solutions are vectors in $\RR^N$
and are also denoted $\tiluh$, $\tilvh$, etc. Then the corresponding 
counterparts of operators action on these vectors are denoted by $\wcalAt, \wcalIt,$ etc.

\paragraph{\it \underline{The finite difference approximation}}
In this case the approximation $\tiluh \in \R^N$  of $u$ is given by
\beq
\wcalAt^{\alpha} \tiluh = \cI_h f := \tilfh, \ \ \mbox{or equivalently} \ \ \ 
\tiluh = 
\wcalAt^{-\alpha}  \tilfh, 
\label{fda}
\eeq
where $\wcalAt$ is an $N\times N$ symmetric and positive definite matrix
coming from  a finite difference approximation to the differential
operator appearing in \eqref{strong}, $\tiluh$ is the vector in $\R^N$ of the approximate solution at the 
interior $N$ grid points,  and $\cI_h f:= \tilfh \in \R^N$ denotes the vector of the values of
the data $f$ at the grid points.  Examples of such matrices are given in Subsection \ref{ss:FD-examples}.

\paragraph{\it \underline{The finite element approximation}}
The approximation in the finite element case is defined in terms of a
conforming finite dimensional space $V_h\subset V$ of piece-wise linear 
functions over a quasi-uniform partition ${\mathcal T}_h$ of $\Omega$ 
into $d$-simplices (intervals, triangles, and simplices in 1-D, 2-D, and 3-D, respectively).
Note that the  construction \eqref{bal}  of negative fractional powers
carries over to the finite dimensional case,  replacing $V$ and $X$ by
$V_h$ with $a(\cdot,\cdot)$ and $(\cdot,\cdot)$ unchanged.  

The discrete operator $\calAt$ is defined to be the inverse of
$T_h:V_h\rightarrow V_h$ with
$T_h g_h:=v_h$  where $v_h\in V_h$ is the unique solution to
\beq
a(v_h,\theta_h) = (g_h,\theta_h),\Forall \theta_h\in V_h.
\label{T_h}
\eeq
The finite element approximation $ u_h \in V_h$ of $u$ is then given by
\beq
\calAt^{\alpha} u_h = \pi_h f, \ \ \mbox{or equivalently} \ \ \ 
u_h = \calAt^{-\alpha} \pi_h f:=\calAt^{-\alpha}f_h,
\label{fea}
\eeq
where $\pi_h$ denotes the $L^2(\Omega) $ projection into $V_h$.  In this
case, $N$ denotes the dimension of the space $V_h$ and equals the number
of (interior) degrees of freedom.  The operator $\calAt$ in the finite element
case is a map of $V_h$ into $V_h$ so that 
$\calAt v_h:=g_h$,  where $g_h\in V_h$ is the unique solution to
\beq
 (g_h,\theta_h)=a(v_h,\theta_h),\Forall \theta_h\in V_h.
\label{calAt-f}
\eeq
Let $\{\phi_j\}$ denote the standard ``nodal" basis of $V_h$.
In terms of this basis  
\beq\label{FEM-matrices}
\calAt \mbox{ corresponds to the matrix } \wcalAt = \wcalMt^{-1} \wcalSt, \ \  
\mbox{where} \ \ \wcalSt_{i,j} =a(\phi_i, \phi_j), \ \ \ \wcalMt_{i,j} =(\phi_i, \phi_j). 
\eeq
In the terminology of the finite element method, $\wcalMt$ and $ \wcalSt$ are the mass 
(consistent mass) and  stiffness matrices, respectively.

Obviously, if $\theta =\calAt \eta$ and $\widetilde \theta,\widetilde \eta \in \R^N$ are the coefficient vectors
corresponding to $\theta,\eta\in V_h$, then $\widetilde \theta = \wcalAt
\widetilde \eta$.  Now, for the coefficient vector $\tilfh$ corresponding to $f_h=\pi_hf$ we have 
 $\tilfh = \wcalMt^{-1} \tilF$, where $\tilF$ is the vector with entries
$$
{\tilF}_j=(f,\phi_j), \qquad \hbox{ for }j=1,2,\ldots,N.
$$
Then using vector notation so that $\tiluh$ is the coefficient vector representing the solution $u_h$ 
through the nodal basis, we can write the finite element approximation of \eqref{strong} 
in the form of system 
\beq \label{classic-FEM}
\wcalAt \tiluh = \wcalMt^{-1} \tilF \ \ \mbox{which implies} \ \ \calSt \tiluh = \tilF.
\eeq
Consequently, the finite element approximation of the sub-diffusion problem \eqref{fea} becomes
\beq\label{mat-FEM}
\wcalMt \wcalAt^\alpha \tiluh = \tilF  \quad \hbox{or equivalently }\quad \tiluh=
\wcalAt^{-\alpha} \wcalMt^{-1} \tilF.
\eeq
\vspace{2mm}

\paragraph{\it \underline{The lumped mass finite element approximation}}

We shall also introduce the finite element method with ``mass lumping" for two reasons.
First, it leads to positivity preserving fully discrete methods  (see, Section \ref{ss:fem-lumped}).
Second, it is well known that lumped mass schemes for linear elements on uniform rectangular 
meshes are equivalent to the simplest finite difference approximations. In fact, as shown later, 
the matrix \eqref{FD-matrix-1D} of the
finite difference approximation of 1-D problem is the same as the matrix of the 
lumped finite element method for linear elements.  Therefore, the theoretical study of
the lumped mass method answers the question about the convergence of the finite difference method 
for solving the problem \eqref{frac-eq}, an outstanding issue in this area.

We introduce the lumped mass (discrete) inner product $ (\cdot,\cdot)_h$ on $V_h$ in
following way (see, e.g. \cite[pp.~239--242]{Thomee2006})  for $d$-simplexes in $\R^d$:
\beq\label{mass-lumping}
(z,v)_h = \frac{1}{d+1} \sum_{\tau \in {\mathcal T}_h } \sum_{i=1}^{d+1} |\tau|  z(P_i) v(P_i)
\ \ \mbox{and  } \ \  {\wcalMt}_h =\{ (\phi_i,\phi_k)_h\}_{i,k}^N.
\eeq
Here  $P_1, \dots, P_{d+1}$ are the vertexes of the $d$-simplex $\tau$ and $|\tau|$ is its 
$d$-dimensional measure.  The matrix $\wcalMt_h$ is called lumped mass matrix.
Simply, the ``lumped mass" inner product 
is defined by replacing the integrals determining the finite element mass matrix by local 
quadrature approximation, specifically, the quadrature defined by
summing values at the vertices of a triangle weighted by the area of the triangle. 

In this case, we define $\calAt$ by 
$\calAt v_h:=g_h$  where $g_h\in V_h$ is the unique solution to
\beq
 (g_h,\theta_h)_h=a(v_h,\theta_h),     \Forall \theta_h\in V_h
\label{calAt-fm}
\eeq 
so that 
\beq\label{A-lumped}
\calAt   \mbox{  corresponds to the matrix } \wcalAt = {\wcalMt}_h^{-1} \wcalSt,  
\quad \mbox{where} \quad {\wcalMt}_h =\{ (\phi_i,\phi_k)_h\}_{i,k}^N. 
\eeq
Here  $ {\wcalMt}_h$ is the lumped mass matrix which is diagonal with
positive entries.
We also replace $\pi_h$ by $\calI_h$ so that the lumped mass semi-discrete
approximation is given by
\beq u_h = \calAt^{-\alpha} \calI_h f := f_h\quad \hbox{or} \quad   \tilde u_h = 
\wcalAt^{-\alpha} \widetilde F.
\label{lumped-semi}
\eeq
Here $\widetilde F$ is the coefficient vector in the representation of the function $\calI_h f$
with respect to the nodal basis in $V_h$.
We shall call $\tiluh$ in \eqref{fda} and $ u_h$ in \eqref{fea} and
\eqref{lumped-semi}   {\it semi-discrete   approximations} of $u$.

We note that the matrix $\wcalAt \in \R^{N \times N}$, $N=O(h^{-d})$,
produced by the standard finite element or finite difference method
is positive definite, large, sparse, with a condition number growing like $h^{-2}$ as 
$h \to 0$.  We shall assume in this paper, that the systems of the type 
$(\wcalAt +c \calIt) \tilde u_h = \tilde f_h$,
$c \ge 0$ and $u_h, f_h \in \R^N$ can be solved approximately in an optimal way, namely, by an algorithm  
that requires $O(N)$ arithmetic operations. This could be achieved by using fast solution methods based on
multi-grid, multi-level, domain decomposition, or other techniques.  The aim of our paper is to construct 
a solution method for \eqref{lumped-semi} with optimal computational complexity $O(N)$.

We also note that fractional powers of a symmetric and positive definite matrix are well defined by
matrix diagonalization so 
we can write 
$$
\wcalAt = \Xi^t  \Lambda \Xi
$$
with $\Xi$ an orthogonal matrix and $\Lambda$ a diagonal matrix with
entries, $\Lambda_{ii}=\lambda_i$ where $0<\lambda_1\le
\lambda_2\le \cdots \le \lambda_N$ are the eigenvalues of $\wcalAt$.  In this case,
\begin{equation}\label{eq:spectral}
\wcalAt^{-\alpha} =  \Xi^t  \Lambda^{-\alpha}  \Xi.
\end{equation}
Of course, $\Lambda^{-\alpha}$ is a diagonal matrix with
diagonal entries $\lambda_i^{-\alpha}$, $i=1,2,\ldots,N$.

The direct computation of $\tiluh$ involves the computation of the eigenvalues and 
eigenvectors of the matrix $\wcalAt$. Such computation using this
factorization is, generally, quite expensive, except for a
very narrow class of equations  with constant coefficients on
rectangular domains.   Similar techniques can be employed in both the
standard mass and lumped mass finite element cases, but requires
expansion in a basis of eigenvectors satisfying generalized eigenvector
problems involving the matrices $\wcalMt$, $\wcalMt_h$ and $\wcalSt$ and,
again, direct computation is quite expensive.

Nevertheless, such approach could be made quite efficient 
for approximation of the corresponding elliptic operator by a spectral
numerical method in simple domains, e.g. 
\cite{Song2017ComputingFL} .
For such problems the spectral methods are known to be very accurate due to exponential convergence
rate with respect to the number of the degrees of freedom. Such examples on square 
domains
are presented in  \cite{Song2017ComputingFL}. Alghough, the case of
spectral approximation on a disk domain is discussed there, 
the application of their discretization to the fractional
power problem  would
require  computing the generalized eigenvectors for which fast methods
are not available.  In contrast, their discretization would  be an ideal candidate for the
method discussed here and  only limited by the avaliability of fast solvers
for the stationary problem.  
In our paper, the targeted area is a steady state problem in a complex domain with low regularity solution
discretized by standard finite element or finite difference method, naturally leading to large scale
linear systems.

\subsection{Fully discrete schemes based on the best uniform rational approximation}
\label{ss:fullydiscrete}

Here we will introduce approximations of $u_h =\calAt^{-\alpha} f_h$ by
employing best rational approximations (BURA) to $t^\gamma$ on $[0,1]$ with $\gamma>0$.
Specifically, we consider BURA along the diagonal of
the Walsh table and take $\rat k$ to be the set of rational functions of
the form  $P_k(t)/Q_k(t)$ with $P_k(t)$ and $Q_k(t)$
polynomials of degree $k$ and $Q_k(0)=1$.   The best rational
approximation (BURA) of $t^\gamma$ is the rational function $r_{\gamma,k}\in \rat k$ 
satisfying
\begin{equation}\label{bura}
 r_{\gamma,k}(t) := \argmin_{s(t)\in \rat k}\,  
 \|  s(t) - t^{\gamma} \|_{L^\infty[0,1]}.
\end{equation}
Denoting the error by  
$$
E_{\gamma,k}:=\|r_{\gamma,k} (t) - t^{\gamma} \|_{L^\infty[0,1]},
$$
we apply Theorem 1 of \cite{Stahl93} to claim that there is a constant $C_\gamma>0$,  independent of $k$, such that 
\begin{equation}\label{rat-error}
E_{\gamma,k}  \le C_\gamma e^{-2 \pi \sqrt{k \gamma}}.
\end{equation}
Thus, the BURA error converges exponentially to zero as $k$ becomes large.

\paragraph{\it Rescaling and the semi-discrete approximation}
We rescale the equations \eqref{fda},  \eqref{fea} and \eqref{lumped-semi}:
\beq\label{uh}
\tiluh= \lambda_1^{-\alpha} (\lambda_1 \wcalAt^{-1})^{\alpha} \tilfh
\eeq
where $\lambda_1 $ denotes the smallest  eigenvalue of $\wcalAt$ in 
\eqref{fda}, \eqref{fea} and \eqref{A-lumped}, respectively. 
The scaling by $\lambda_1$ maps the eigenvalues of $\lambda_1 \wcalAt^{-1}$ 
to the  interval $(0,1]$.

We note that  instead of scaling with $\lambda_1 $,  we can scale with any $\delta \in (0,\lambda_1]$.
In this case the eigenvalues of $\delta   \wcalAt^{-1}$ will be again in the interval $(0,1]$ and the method will work
in the same way. This will allow to use any lower bound for  the eigenvalue $\lambda_1$. 
Such a bound could be obtained using the coercivity of the form $a(\cdot,\cdot)$ in $V$
and the Poincar\'e-Friedrichs inequality. For example, 
if $\calAt$ is obtained by finite element discretization of 
$-\Delta$ in a bounded, convex, Lipschitz domain $\Omega \subset \R^d$ with a diameter 
$d_{\Omega}=diam(\Omega)$ with Dirichlet boundary conditions, then  
the bound $\lambda_1 \ge C_P = { \pi^2 d}/{ d_\Omega^2}$ 
follows from the Poincar\'{e}-Friedrichs inequality:
$ C_P \|u\|_{L^2(\Omega)} \le \| \nabla u \|_{L^2(\Omega)}  $ for all $u \in H^1_0(\Omega)$,  \cite[inequality (1.9)]{PW_60}. 
 Another possibility is to find an estimate for $\lambda_1$ by running  few iterations 
of the inverse power method.

Now we introduce the {\it \underline{fully discrete}} approximations:  
$ w_h  \in V_h$ of the finite element approximation $u_h \in V_h$  and 
$\tilwh \in \R^N$ of the finite difference approximation $\tiluh  \in \R^N$ by
\beq
w_h 
=\lambda_1^{-\alpha} r_{\alpha,k} (\lambda_1 \calAt^{-1}) f_h
\quad \mbox{and}  \quad 
\tilwh 
=\lambda_1^{-\alpha} r_{\alpha,k} (\lambda_1 \wcalAt^{-1}) \tilfh.
\label{wh}
\eeq
Here $\calAt$ and $f_h$ are as in \eqref{fea} or \eqref{lumped-semi}
and $\wcalAt$ and $\tilfh$ are as in \eqref{fda}.

In Section \ref{sec:implement}, we study the error of these fully discrete solutions.
For the finite element case we obtain the error estimate 
\beq
\|u_h - w_h\| \le \lambda_1^{-\alpha} E_{\alpha,k} \|f_h\|
\label{pbest}
\eeq
with $\|\cdot\|$ denoting the norm in $L^2(\Omega)$.
In the finite difference case,
we have
\beq
\|\tiluh - \tilwh  \|_{\ell_2} \le \lambda_1^{-\alpha} E_{\alpha,k} \|\tilfh \|_{\ell_2}
\label{pbest-fd}
\eeq
where the norm $\| \cdot \|_{\ell_2} $ 
denotes the  Euclidean norm in $\RR^N$.

We note that the schemes of \cite{HLMMV18} 
are closely related
to our scheme.  These were given for the finite difference case by writing
$$
\tiluh=\wcalAt^{-p} (\wcalAt^{p-\alpha}) \tilfh = \lambda_N^{\alpha-p} \wcalAt^{-p}
(\wcalAt/\lambda_N)^{p-\alpha} \tilfh, \quad p=1,2.
$$
Their approximation becomes
\beq
\tiluh=\lambda_N^{p-\alpha} \wcalAt^{-p} r_{p-\alpha,k}(\wcalAt/\lambda_N) \tilfh    
\quad \mbox{  that implies  } \quad  \|\tiluh - \tilwh \|_{\ell_2} \le \lambda_N^{p-\alpha} E_{\alpha,k} \| \tilfh \|_{\ell_2} .
\label{orig}
\eeq
The main disadvantage of this method compared to ours is that
$\lambda_N$ grows on the order of $h_{min}^{-2}$ with $h_{min}$ denoting
the minimal distance between mesh points so the factor of
$\lambda_N^{p-\alpha}$ deteriorates the convergence rate.  
This is especially harmful when local mesh refinement is used.
 In contrast, $\lambda_1$ is
related to the constant in the Poincar\'e inequality and remains bounded
away from zero independently of the mesh parameter so the appearance of
$\lambda_1^{-\alpha}$ in our method is harmless.

\paragraph{\it Existing solution methods for fractional powers of SPD matrices}

Due to the current  interest of the computational mathematics and physics communities 
in modeling and simulations involving fractional powers of elliptic operators, 
a number of approaches and algorithms has been developed, studied, and tested on various problems, 
 \cite{aceto2018efficient,nochetto2015pde,bonito2018sinc,BP15,nochetto2016pde,HOFREITHER2019}.
 However, the goal of this paper is to develop efficient methods for solving 
large systems (hundreds of thousands or even minions of unknowns) algebraic equations \eqref{fda} that 
utilize efficient methods for solving the system $\wcalAt \tiluh = \tilfh$.
 Below we make a concise survey of such methods.

(1) In the finite difference case, $\tiluh$ is expressed though a
fractional power of a symmetric and positive definite matrix.   
We can look at this problem as a particular case of  the well established methods of 
stable computations of the matrix square root or/and other functions of
matrices, see, e.g. \cite{druskin1998extended,Higham1997,Kenney1991}.
Often these are based on Newton iteration with suitable Pad\'e stabilization. 
Application of such approach is limited to small-size matrices.

(2) An extension of the problem from $ \Omega \subset \R^d$  to a problem in 
$\Omega \times (0,\infty) \subset \R^{d+1}$, see, e.g. \cite{caffarelli2007extension}. Nochetto and co-authors 
in \cite{nochetto2015pde,nochetto2016pde}  developed efficient computational method based on  finite element 
discretization of the extended problem and subsequent use of multi-grid technique. The main deficiency of the 
method is that instead of problem in $\R^d$, one needs to work in a domain in one dimension higher which adds to the 
complexity of the developed algorithms.

(3) Reformulation of the problem as a pseudo-parabolic on the cylinder $(0,1) \times \Omega$
by adding a time variable $t \in (0,1)$. Such methods were proposed, developed, and tested by Vabishchevich in
\cite{vabishchevich2015numerical,vabishchevich2018numerical}. As shown in the numerical experiments in 
\cite{HOFREITHER2019}, the method is very slow when using uniform time stepping. 
However, the improvement proposed in \cite{DuanLazarovPasciak,CIEGIS2019} make this method quite competitive.

(4) Approximation of the Dunford-Taylor integral representation of the solution of equations involving fractional powers of 
elliptic operators, proposed in the pioneering paper of Bonito and Pasciak \cite{BP15}. Further the idea was extended and 
augmented in various directions in \cite{Aceto_17,BP17,bonito2019sinc,bonito2019numerical}. These methods use exponentially 
convergent sinc quadratures.

(5) Best uniform rational approximation of the function $t^\alpha$ on $[0,1]$, proposed in \cite{HMMV2016,HLMMV18},
further developed in \cite{harizanov2018positive,harizanov2019cmwa,harizanov2019analysis} 
and called BURA methods. In this paper we propose, analyze, and test a new method given by \eqref{wh} that is
based on rescaling the problem with the smallest eigenvalue of matrix $\wcalAt$ (or the operator $\calAt$).


As shown recently in \cite{HOFREITHER2019},  in appearance different,
these methods  are interrelated and all seem to involve some rational approximation of
fractional powers of the underlying elliptic operator, see, e.g. \cite{Aceto_17,bonito2018sinc,BP15}. 
In the mentioned above works the 
numerical algorithm results in a rational approximation of $ \calA^{-\alpha} f$
where the elliptic operator  $\calA$ is replaced by some approximation by finite elements. 
The algorithm we propose and study in this paper is based on the best uniform 
rational approximation and in principle should be at least as good as any of these methods.
In fact, our comparisons show that in many cases the proposed method performs  
significantly better.

However, one should realize that BURA-based  methods involve 
Remez method of finding the best uniform rational approximation
by solving the highly non-linear min-max problem \eqref{bura}.  
It is well known that Remez algorithm is very sensitive to the precision of the computer arithmetic, 
cf. \cite{PGMASA1987,varga1992some,Driscoll2014}. One of the main reasons is that almost all extreme points
of the error function tend to the origin
as $k\to\infty$, cf. \cite[Theorem 4]{SS93}. Various techniques for stabilization
of the method have been used, mostly by using Tchebyshev orthogonal polynomials, cf.
\cite{Driscoll2014}. It is observed that to achieve high accuracy one needs to use high arithmetic precision.
For example, in \cite{varga1992some} the first 25 correct decimal digits of the BURA error of
$t^\alpha$ for six values of $\alpha\in [0,1]$ are reported for
degree up to $k=30$ by using computer arithmetic with 200 significant digits.

\paragraph{\it Positivity preserving schemes}

 In the {\it finite difference} case,  if $f(x)\ge 0 $ for $x\in \Omega$
then the vector $\tiluh$ defined by \eqref{fda} has non-negative
entries.  The issue of positivity preservation of approximate solution of problems involving the spectral
fractional Laplacian by finite difference method was first discussed and established for the 
BURA scheme \eqref{orig} in \cite{harizanov2018positive}.
We will show that the solution $\tilwh$ of \eqref{wh}  has non-negative entries as well.

However, in the {\it finite element} case, we will show that $u_h(x)$ can have negative values
even when $f$ is non-negative, i.e. the consistent finite element approximation may loose
non-negativity of the solution.  Instead, we shall use  schemes obtained by mass lumping discussed in
Subsection \ref{ss:semi-discrete}.  Note that in \eqref{fea} we 
also replace $\pi_h$ by the interpolant  $\calI_h:C^0(\Omega)\rightarrow V_h$.  This leads to
a semi-discrete solution 
\beq
u_h= \calAt^{-\alpha} \calI_h f = \lambda_1^{-\alpha} (\lambda_1\calAt^{-1})^\alpha \calI_h f
\label{massu}
\eeq
and its BURA approximation (fully discrete approximation)
\beq
w_h= \lambda_1^{-\alpha} r_{\alpha,k}(\lambda_1\calAt^{-1}) \calI_h f
\label{massw}
\eeq
with $\calAt$ defined by \eqref{A-lumped}.  
In this case, $w_h$ is non-negative when $f$ is non-negative
and most of the approximation properties of $u_h$ are still preserved
(see, Theorem~ \ref{masslumped}).
 
\subsection{Organization of the paper and our contributions}
\label{ss:contributions}

Section \ref{sec:implement} examines the implementation of \eqref{wh} for both, the finite 
element \eqref{fea} and finite difference approximations \eqref{fda}, and also proves the estimates \eqref{pbest} 
and \eqref{pbest-fd} for their error.  Here we also consider the lumped mass method and discuss the 
non-negativity of the solution produced by non-negative data.
In Section \ref{sec:FD} we give several matrices obtained by finite difference method and 
perform some extensive computations on a number of test problems in one and two 
spatial dimensions. We compare the accuracy of 
the proposed in this paper new method, called P-BURA, with the BURA 
method \eqref{orig} of \cite{HLMMV18} and with the method of Bonito and Pasciak, \cite{BP15},
on two 2-dimensional model problems with smooth (Table \ref{tab:2D results 2}) 
and non-smooth right hand sides (Table \ref{tab:2D results 1}). Further in Section  \ref{sec:FD}
we study the efficiency of the method on some non-uniform meshes refined locally 
in order to capture the interior layers of the solution. The results are reported  in
Tables \ref{tab:1Dconst} and  \ref{tab:1Dsource}.
Section  \ref{sec:lumped}  focuses on the finite element approximations. 
In Theorem \ref{masslumped} we provide error
estimates for $u-u_h$ both in the consistent mass finite element method (cf. \cite{BP15}) and
the case of mass lumping.  As a consequence, in Corollary \ref{FD-convergence} we establish
an error bound for the finite difference approximation of boundary value 
problem for the spectral fractional elliptic equation.  To best of our knowledge, 
error bounds for the approximations of the spectral fractional Laplacian by finite 
differences is not available.

The main contributions of this paper are: (1) derivation an analysis of an efficient BURA 
method for solving systems of equations $\wcalAt^\alpha \tiluh = \tilfh$, $0< \alpha <1$,  in $\R^N$, where
$ \wcalAt$ is a symmetric and positive sparse matrix obtained from finite difference or finite element 
approximations of elliptic operators; (2)  analysis of lumped mass schemes that lead to positivity 
preserving methods; (3) estimates of approximation error of the spectral fractional Laplacian by finite 
differences.

\section{Implementation and basic estimates of the error}
\label{sec:implement}

\subsection{Properties of the best uniform rational approximation}
\label{ss:BURA}

In this section, we discuss the implementation of \eqref{wh} in the
finite difference and finite element cases.

First, on Table \ref{t:error} 
we present the computed error of BURA $r_{\alpha,k}(t)$ of $t^\alpha$ using the
modified Remez algorithm, e.g. \cite{HLMMV18}. As expected, the approximation error 
for large $\alpha$ is in the very reasonable 
range of $10^{-5} - 10^{-7}$ for 
$k=7 - 10$. Moreover, for this range of $k$ the Remez algorithm is relatively stable and 
the coefficients of BURA function $r_{\alpha,k}(t)$ are determined with good accuracy.
\begin{table}[h!]
\caption{ Errors $E_{\alpha,k}$ of $r_{\alpha,k}(t)$ for $t\in[0,1]$, used in BURA 
and P-BURA computations.}\label{t:error}
\centering
\begin{tabular}{|c|c|c|c|c|c|c|}
\hline
$\alpha$&$E_{\alpha,5}$&$E_{\alpha,6}$&$E_{\alpha,7}$&$E_{\alpha,8}$&$E_{\alpha,9}$&$E_{\alpha,10}$\\ \hline
  0.75 & 2.8676e-5 & 9.2522e-6 & 3.2566e-6 & 1.2288e-6 & 4.9096e-7 & 2.0584e-7\\
  0.50 & 2.6896e-4 & 1.0747e-4 & 4.6037e-5 & 2.0852e-5 & 9.8893e-6 & 4.8760e-6\\
  0.25 & 2.7348e-3 & 1.4312e-3 & 7.8650e-4 & 4.4950e-4 & 2.6536e-4 & 1.6100e-4\\\hline
\end{tabular}
\end{table}

Next, we prove the estimates \eqref{pbest} and \eqref{pbest-fd}.
It is known that the best rational approximation
$r_{\alpha,k}(x)=P(x)/Q(x)$ 
for
$\alpha \in (0,1)$ is
non-degenerate, i.e., the polynomials $P$ and $Q$ are of full
degree.   Let the roots of $P$ and $Q$ be denoted by
$\zeta_1, \dots, \zeta_k$ and  $d_1, \dots, d_k$, respectively.    
It is shown in
\cite{SS93,stahl2003}  that the roots interlace and satisfy
\begin{equation}\label{interlacing}
  0 > \zeta_1 > d_1 > \zeta_2 > d_2 > \cdots > \zeta_k > d_k.
\end{equation}
We then have
\beq
r_{\alpha,k} (t)=b\prod_{i=1}^k\frac{t-\zeta_i}{t-d_i}\label{prodr}
\eeq
where, by \eqref{interlacing} and the fact that
$r_{\alpha,k}$ is a best approximation to a non-negative function, $b>0$ and
$P(x)>0$ and $Q(x)>0$ for $x\ge 0$.

\def\trg{\tilde r_{\alpha,k}}

We consider $\trg$ defined by
$$\trg(\lambda):=r_{\alpha,k}(1/\lambda)=\frac {\widetilde
  P(\lambda)}{\widetilde Q(\lambda)}.$$
Here $\widetilde P(\lambda)=\lambda^k P(\lambda^{-1})$ and $\widetilde Q(\lambda)= \lambda^k
Q(\lambda^{-1})$  and hence their coefficients 
are defined by reversing
the order of the coefficients in $P,Q$ appearing in $r_{\alpha,k}$.  
In
addition, \eqref{interlacing} implies
\begin{equation}\label{interlacing1}
  0 >  \tilde d_k > \tilde \zeta_k > \tilde d_{k-1} > \tilde
  \zeta_{k-1}\cdots >   \tilde d_1>\tilde \zeta_1.
\end{equation}  Here $\tilde d_i=1/d_i$ and $\tilde \zeta_i = 1/\zeta_i
$ are the roots of $\widetilde P$ and $\widetilde Q$, respectively.

\begin{proposition}\label{lemma:BURA}  For $\alpha\in (0,1)$, 
\beq
  \trg(\lambda)=c_0+\sum_{i=1}^k \frac{c_i}{\lambda-\tilde d_i}
  \label{trg}
  \eeq
  where $c_i>0$ for $i=0,1,\ldots,k$.
\end{proposition}

\begin{proof} We note that
  $$  
  c_0=\lim_{\lambda\rightarrow \infty} \trg(\lambda)
 =  \lim_{x\rightarrow 0} r_{\alpha,k} (x)=
  b\prod_{i=1}^k{\zeta_i}/{d_i}>0.
  $$
  The remaining coefficients in \eqref{trg} are
  determined by the equation
  $$
  \widetilde P(\lambda) =\widetilde Q(\lambda)\bigg [ c_0+\sum_{i=1}^k
    \frac{c_i}{\lambda-\tilde d_i}\bigg ]
  $$
  which when evaluated at $\tilde d_i$ implies 
  \beq
  \widetilde P(\tilde d_i) = c_i \prod_{j\neq i}  (\tilde d_i-\tilde d_j).
  \label{tprod}
  \eeq
  It follows from \eqref{interlacing1} that both the signs of
  $\widetilde P(\tilde d_i) $ and those of the product on the 
  right hand side of \eqref{tprod} oscillate with $i$.  In addition, \eqref{interlacing1}
  also implies that the product in \eqref{tprod} is positive for $i=k$
  and the sign of $\widetilde P(\tilde d_k)$ is the same as that of
  $$\widetilde P(1) = P(1) >0.$$
  It follows that $c_i>0$, for $i=1,2,\ldots,k $.
\end{proof}

Now consider the implementation of \eqref{wh}.  By \eqref{mat-FEM}, the
coefficients vector $\tilwh$ of the representation of $w_h$ with respect to the nodal basis in $V_h$ is given by
\beq
\tilwh 
=\lambda_1^{-\alpha} \widetilde r_{\alpha,k} (\lambda_1^{-1} \wcalAt) 
\wcalMt ^{-1} \widetilde F\label{tilwh}
\eeq
with $\wcalAt $ denoting the matrix in \eqref{FEM-matrices}.   Applying 
Proposition~\ref{lemma:BURA} gives
\beq
\bal 
\tilwh &= \lambda_1^{-\alpha}\bigg( c_0 \wcalMt ^{-1} \widetilde F +\sum_{i=1}^k (\lambda_1 c_i)( \wcalAt -\lambda_1 \tilde
  d_i\wcalIt)^{-1} \wcalMt ^{-1} \widetilde F\bigg)\\
&= \lambda_1^{-\alpha}\bigg( c_0 \wcalMt ^{-1} \widetilde F +\sum_{i=1}^k (\lambda_1 c_i)( \wcalSt -\lambda_1 \tilde
  d_i\wcalMt)^{-1} \widetilde F\bigg).
\eal
\label{sumwh}
\eeq
We note that even though $\widetilde F$ is a non-negative  vector when
$f$ is non-negative, the matrix  $( \wcalSt -\lambda_1 \tilde
  d_i\wcalMt)^{-1}$ is not positivity preserving when $-\lambda_1 d_i$
  is large.    This is because even if $\wcalSt$ is an $M$-matrix,
  $\wcalSt +\gamma \wcalMt$ for large $\gamma$ 
becomes a matrix with the same
  sparsity pattern where every matrix entry in the pattern is positive.
  The inverse of such a matrix is NOT positivity preserving\footnote{This, in
  turn, implies that the matrix $\wcalAt$ in the finite element case
  CANNOT be an $M$-matrix.}.
As the
  matrix $\wcalMt$  has positive entries, its inverse appearing in the first
  term above is also not positive preserving.

In the finite difference case,  $\tilwh$ satisfies
\eqref{tilwh} with 
$\wcalMt f_h$ replaced by $\tilfh$ and $\wcalAt$ denoting the finite
difference matrix.    Applying 
Proposition~\ref{lemma:BURA} gives
\beq
\tilwh = \lambda_1^{-\alpha}\bigg( c_0 \tilfh +\sum_{i=1}^k (\lambda_1 c_i)( \wcalAt -\lambda_1 \tilde
  d_i\wcalIt)^{-1} \tilfh \bigg).
\label{sumtwh}
\eeq
The finite difference matrix $\wcalAt$ is generally an $M$-matrix and we
have the following theorem (is a consequence of
Proposition~\ref{lemma:BURA}
and \eqref{sumtwh}).

\begin{proposition}\label{lemma:BURA-positive}  
Assume that the finite difference matrix $\wcalAt$ is an M-matrix, i.e. all diagonal
entries are positive and all non-diagonal entries are non-positive.
If $ \tilfh$ has all its entries non-negative then the solution $\tilwh$ 
represented by \eqref{sumwh} has all its entries non-negative, i.e.  
the method is positivity preserving.
\end{proposition}


The above sum is trivially parallelizable as the result of each term is
independent of all others.
Alternatively, by \eqref{prodr},
$$
\tilde r_{\alpha,k} (\lambda) =
b\prod_{i=1}^k\frac{1-\lambda\zeta_i}{1-\lambda d_i}
$$
and hence
\beq
w_h = \lambda_1^{-\alpha}b\bigg[\prod_{i=1}^k(\lambda_1\calIt-\zeta_i\calAt )(\lambda_1\calIt-d_i\calAt)^{-1}\bigg] f_h.
\label{prodwh}
\eeq
This product needs to be computed sequentially, computing the $j$-term
product by applying the $j$'th  operator to the $j-1$'st term product.

We note that both,  the additive \eqref{sumwh} and the multiplicative \eqref{prodwh} versions  of the method 
led to stable computations. Indeed, due to the interlacing properties \eqref{interlacing} 
the $L^2$-norm of $ (\lambda_1\calIt-\zeta_i\calAt )(\lambda_1\calIt-d_i\calAt)^{-1}$ is less than 1 and 
the product computation is stable.
Similarly, since $c_j>0$, $j=0,\dots,k$ the summation in \eqref{sumwh} is stable.

\subsection{Finite difference scheme}
\label{ss:FD}
We first consider the finite difference case.  
In this case, each  term in the sum \eqref{sumwh} requires a sparse matrix solve 
that involves a matrix which is a sum of  $\calAt$ and the scaled (with a positive factor)  
identity.  The sequential computation is similar with each step involving a sparse matrix 
multiply and a sparse matrix solve.

We next show that \eqref{pbest-fd} holds.   We note that 
$$\tiluh - \tilwh = \lambda_1^{-\alpha} [(\lambda_1 \wcalAt^{-1})^{\alpha} -
r_{\alpha,k} (\lambda_1 \wcalAt^{-1})] \cI_h f:=   \lambda_1^{-\alpha} \calGt
\cI_h f.$$
The above matrix $\calGt = (\lambda_1 \wcalAt^{-1})^{\alpha} - r_{\alpha,k} (\lambda_1 \wcalAt^{-1})$ 
is symmetric and hence 
$$
\| \tiluh - \tilwh  \|_{\ell_2} \le \lambda_1^{-\alpha} \rho(\calGt)
\|\cI_h f \|_{\ell_2}
$$
with $\rho(\calGt)$ denoting the spectral radius of $\calGt$.   The inequality
\eqref{pbest-fd} follows from noting that the eigenvalues of $\calGt$ come
from those of $\calAt$, i.e.,
$$
\rho(\calGt)= \max_{i=1}^N |(\lambda_1/\lambda_i)^\alpha
-r_{\alpha,k}(\lambda_1/\lambda_i)|\le \max_{1 \le t \le \infty} |t^\alpha - r_{\alpha,k}(t)| = E_{\alpha,k}.
$$

\subsection{Consistent mass finite element method} 
\label{ss:FEM}

The implementation of finite element problems is done in terms of
stiffness and mass matrices denoted  by $\wcalSt$ and $\wcalMt$, respectively, and vectors 
in $\RR^N$ where $N$ is the dimension of $V_h$ defined by \eqref{FEM-matrices}.
Then the coefficients $\tilwh$ for the function $w_h$ using \eqref{sumwh}  are given by
\beq
\tilwh 
= \lambda_1^{-\alpha}\bigg(c_0 \wcalMt^{-1} \widetilde F 
  +\sum_{i=1}^k (\lambda_1 c_i)( \wcalSt -\lambda_1 \tilde d_i\wcalMt)^{-1} \widetilde F\bigg).
\label{sumwhfe}
\eeq
Similarly,  for the product case  we get  
\beq
\tilwh = \lambda_1^{-\alpha}b
 \bigg[\prod_{i=1}^k 
(\lambda_1\wcalMt-d_i\wcalSt)^{-1}(\lambda_1\wcalMt-\zeta_i\wcalSt)\bigg] \wcalMt^{-1} \widetilde F.
\label{prodwhfe}
\eeq
 The matrices in parenthesis appearing both in \eqref{sumwhfe} and \eqref{prodwhfe} are
positive linear combinations of the symmetric and positive definite
stiffness and mass matrices.

The validation of \eqref{pbest} in the finite element case is similar.
Let $\{\psi_1,\psi_2,\ldots,\psi_N\}$ be an  $L^2(\Omega)$ orthonormal basis of eigenfunctions in
$V_h$ with
eigenvalues $0<\lambda_1\le \lambda_2\le \cdots \le \lambda_N$
satisfying the generalized eigenvalue problem:
$$
a(\psi_i,\theta)= \lambda_i(\psi_i,\theta),\Forall \theta\in V_h.
$$
Note that $\calAt$  defined through $(\calAt \psi_i, \theta)= a(\psi_i,\theta)$    
has matrix representation  
$\wcalMt^{-1} \calSt$ so that this eigenvalue problem 
has equivalent algebraic form $\calSt {\widetilde \psi_j} = \lambda_j \wcalMt {\widetilde \psi_j}$. 
Expanding $\pi_h f$, $w_h$ and $u_h$ in this  basis leads to 
$$u_h-w_h = \lambda_1^{-\alpha} \sum_{i=1}^n [(\lambda_1 /\lambda_i)^{\alpha} -
r_{\alpha,k} (\lambda_1/\lambda_i)]\,  (f,\psi_i) \,\psi_i.$$
The quantities in brackets above are bounded in absolute value by
$E_{\alpha,k}$ and the inequality \eqref{pbest} follows from Parseval's
formula, i.e.,
\beq\label{BURA-error}
\|u_h-w_h\|^2\le \lambda_1^{-\alpha}E_{\alpha,k}^2 \sum_{i=1}^n\
(f,\psi_i)^2 =
 \lambda_1^{-\alpha}E_{\alpha,k}^2\|\pi_h f\|^2
 \eeq
and \eqref{pbest} follows.

Though the consistent mass matrix $\wcalMt$ has the same sparsity pattern as the 
stiffness matrix it is not an $M$-matrix. Then the
formula \eqref{sumwhfe} shows that even when $\calSt$ is an $M$-matrix, the 
semidiscrete solution $\tilwh$ may not be non-negative for $f(x) \ge 0$
since $\wcalSt+\mu \wcalMt$ fails to be an $M$-matrix for large $\mu$.  
The issue of positivity preservation is discussed in more details and 
illustrated with numerical examples in Subsection \ref{ss:consistmass}.

\subsection{Lumped mass finite element method}
\label{ss:fem-lumped} 

Since in this time 
$\calAt $ has matrix representation $  \wcalMt_h^{-1} \calSt$ then \eqref{sumwhfe} becomes 
\beq
\tilwh 
= \lambda_1^{-\alpha}\bigg(c_0 \wcalMt_h^{-1} \widetilde F 
     +\sum_{i=1}^k (\lambda_1 c_i)( \wcalSt -\lambda_1 \tilde d_i\wcalMt_h)^{-1} \widetilde F\bigg).
\label{sumwhfe-lump}
\eeq
The analysis of this scheme is the same as the analysis of the standard FEM scheme.  The
only difference is that now we need to 
use the eigenvalues and eigenfunctions of 
$ a(\psi_i,\theta)= \lambda_i(\psi_i,\theta)_h$ for all $\theta\in V_h$,
and the analysis follows easily.

The main purpose of introducing the lumped mass method is to ensure 
non-negativity of the fully discrete solution  $ \tilwh $ 
in case of non-negative data $f$. 
Due to \eqref{interlacing1} and Proposition \ref{lemma:BURA}
 we have  $\tilde d_i <0$, and $c_i >0$ for $i=1, \dots, k$.
Matrix $\wcalMt_h$ is diagonal with positive elements  and 
representation \eqref{sumwhfe-lump}  shows that 
if $\wcalSt$ is an $M$-matrix,  then $   \wcalSt -\lambda_1 \tilde d_i\wcalMt_h$, $i=1, \dots,k$
will be all $M$-matrices and the fully discrete solution will satisfy $ \tilwh  \ge 0 $ if $ f \ge 0$.
Thus, to ensure non-negativity it is sufficient
the stiffness matrix $\calSt$ to be an $M$-matrix. 
This phenomenon is well understood in the case considered in this paper, namely,  conforming linear
finite elements on simplicial meshes. For  $a(x)=1$ and $d=2$ in \cite{ciarlet-raviart} this was shown to hold
provided that the mesh triangles do not have any angles exceeding $\pi/2$. 
The most general result  is established in \cite[Lemma 2.1]{xu_zikatanov1999}, namely, 
a sufficient and necessary condition $\wcalSt$ to be an $M$-matrix
for simplicial meshes in $\R^d$ for any $d \ge 2$. The condition is expressed 
through the angles between faces and areas of the simplex faces and improves
the result from  \cite{ciarlet-raviart} for $d=2$.

\section{
Finite difference approximation of the fractional diffusion problem} 
\label{sec:FD}

The linear operators we consider in this section are approximations of \eqref{strong}
by finite differences. We begin with some simple examples.

\subsection{Example of Finite Difference Approximations}\label{ss:FD-examples}
Now we give two particular examples of finite difference approximations of elliptic operators. These
are used to illustrate the above theory and are also a basis of our numerical experiments.

\paragraph{\it Example 1} We first consider the one-dimensional equation \eqref{strong} with
variable coefficient, namely, we study the following boundary value problem
$ - ( a(x) u^{\prime})^{\prime} =f(x),$ $ u(0)=0, \ u(1)=0, \ $ for $ 0<x<1$,
where $a(x)$ is uniformly positive function on $[0,1]$. On a uniform mesh $x_i =ih$, $i=0, \dots, N+1$,
$ h=1/(N+1)$,  we consider the three-point  approximation of the second derivative
\begin{equation*}
\begin{split}
 u''(x_i) &
\approx 
 \frac{1}{h}\left (a_{i+\frac12}\frac{u(x_{i+1}) - u(x_i)}{h} - a_{i-\frac12}\frac{u(x_i) - u(x_{i-1})}{h} \right )
\end{split}
\end{equation*}
Here $a_{i-\frac12}=a(x_i - h/2)$ or  $a_{i - \frac12}=\frac{1}{h}\int_{x_{i-1}}^{x_i} a(x) dx $. Note that the former is 
the standard finite difference approximation obtained from 
the balanced method (see, e.g. \cite[pp. 155--157]{samarskii2001theory}), 
while the latter is a result of finite element method with mass lumping, see Subsection \ref{ss:FEM}.

Then the finite difference approximation in this case is the matrix equation \eqref{fda} with 
\begin{equation}\label{FD-matrix-1D-k} 
\wcalAt= \frac{1}{h^2} \left[
\begin{array}{ccccc} a_{\frac12}+ a_{\frac32}& -  a_{\frac32}  &&&\\
 -  a_{\frac32} &  a_{\frac32} +  a_{\frac52} & -  a_{\frac52}&&\\
\cdots &\cdots &\cdots &\cdots &\cdots\\ 
& -  a_{i-\frac12}& a_{i-\frac12} +  a_{i+\frac12} & a_{i+\frac12} &\\\cdots &\cdots &\cdots &\cdots  &\cdots\\
&&& - a_{N-\frac12} &  a_{N-\frac12} + a_{N +\frac12}
\end{array}\right], \ \ 
  \cI_h f =\tilfh = \left[
\begin{array}{c}
 f(x_1) \\
 f(x_2) \\
 \dots  \\
 f(x_i) \\
\dots \\
 f(x_N)
\end{array} \right ] .
\end{equation}
The eigenvalues $\lambda_i$ of the matrix $\calAt $ satisfy 
$$ 
4 \pi^2 \min_x a(x) \le \lambda_i \le  4 \max_x a(x)/h^2,  \ \  i=1, \dots, N.
$$

\paragraph{\it Example 2}
The next example is for problem \eqref{strong}
on $\Omega=(0,1) \times (0,1)$ on a $(n+1) \times (n+1)$ square mesh.
The standard 5-point stencil finite difference approximation of the Laplace operator gives the 
matrix  $\wcalAt \in \R^{N \times N}$, $N=n^2$, that has the
following block stricture (here $ \wcalAt_{i,i} \in \R^{n \times n}$, $i=1, \cdots, n $ 
and $\wcalIt_n$ is the identity matrix in $\R^n$) 
\begin{equation}\label{FD-matrix} 
\wcalAt=
(n+1)^2 \left[\begin{array}{ccccc} \wcalAt_{1,1} & -\wcalIt_n &&&\\ -\wcalIt_n & \wcalAt_{2,2} & -\wcalIt_n 
&&\\\cdots &\cdots &\cdots &\cdots &\cdots\\ & -\wcalIt_n & \wcalAt_{i,i} & -\wcalIt_n &\\\cdots &\cdots &\cdots &\cdots 
&\cdots\\&&& -\wcalIt_n & \wcalAt_{n,n}\end{array}\right],\quad \wcalAt_{i,i}=\left[\begin{array}{cccc} 
4&-1&&\\-1&4&-1&\\\cdots&\cdots&\cdots&\cdots\\&-1&4&-1\\&&-1&4\end{array}\right].
\end{equation}

This matrix could be obtained by the finite element method applied to triangular meshes that 
generated on triangulations obtained by splitting each rectangle into two 
triangles (by connecting the lower left vertex with the upper right one)  
and using the ``lumped" mass inner product \eqref{mass-lumping}.
Since the mesh is square, all diagonal elements of $ {\wcalMt}^{-1}_h$ are equal to 
$ h^{-2}=(n+1)^{2}$.
Then the operator $  {\calAt}: V_h \to V_h$ is defined as 
$ ({\calAt} u_h, v)_h = a(u_h,v) $ has a matrix representation 
 $\wcalAt= {\wcalMt}_h^{-1} {\wcalSt}$,
see also \cite[Chapter 4, p.~203--205]{ciarlet2002}.

\begin{remark}\label{r:FD}
We note that on an uniform mesh with step-size $h=1/(N+1)$  the matrix  \eqref{FD-matrix}
has the following extreme eigenvalues:
$$
\lambda_1 = 8(n+1)^2 \sin^2 \frac{\pi}{2(n+1)} \approx 2 \pi^2, \ \ 
\lambda_{n^2} = 8(n+1)^2 \sin^2 \frac{\pi n}{2(n+1)} \approx 8 (n+1)^2 = 8h^{-2}.
$$
\end{remark}

\paragraph{\it Example 3} We finally consider the lumped mass 
approximation to the one-dimensional equation  
$-\Delta u := - u^{\prime\prime} =f(x),$ $ u(0)=0, \ u(1)=0, \ $ for $ 0<x<1$.
We use an arbitrary
nonuniform grid $0=x_0 < x_1 < \dots < x_N < x_{N+1}=1$.  This results in 
\begin{equation}\label{FD-matrix-1D} 
\calSt= \left[
\begin{array}{ccccc} \frac{1}{h_1}+\frac{1}{h_2} & -\frac{1}{h_2} &&&\\
 - \frac{1}{h_2} & \frac{1}{h_2} + \frac{1}{h_3} & - \frac{1}{h_3}&&\\
\vdots &\vdots &\vdots &\vdots &\vdots\\ 
& -\frac{1}{h_i} &  \frac{1}{h_{i}} + \frac{1}{h_{i+1}} & -\frac{1}{h_{i+1}} &\\
\vdots &\vdots &\vdots &\vdots  &\vdots\\
&&& -\frac{1}{h_N} & \frac{1}{h_N} +\frac{1}{h_{N+1}} 
\end{array}\right], \ \ 
  \tilfh = \left[
\begin{array}{c}
\widetilde h_1 f(x_1) \\
\widetilde h_2 f(x_2) \\
 \vdots  \\
\widetilde h_i f(x_i) \\
\vdots \\
\widetilde h_{N+1} f(x_N)
\end{array} \right ] .
\end{equation}
where $h_i=x_i -x_{i-1}$ and $\widetilde h_i = \frac12(h_{i+1}+h_i)$.
This is the standard finite difference approximation on this mesh, see \cite[pp. 155--157]{samarskii2001theory},
but does not fit into the earlier discussion of the finite difference case.

\subsection{Numerical tests: set up for comparison with other methods}
\label{sec:section4}

The goal of the numerical tests is to see how the accuracy of 
the computations of various methods is affected by 
the main factors, namely, $\alpha \in (0,1)$, the smoothness 
of the solution $u$, the degree of the polynomials $k$, and 
the mesh-size $h$. Note that for a general mesh, 
with $h_*=\min h$,  the matrix $ \wcalAt$ has spectral 
condition  number $ \kappa(\wcalAt) =O(h_*^{-2})$.

Our first numerical tests are based on 5-point 
finite difference approximation  of the 2-D fractional Laplacian on a uniform square mesh
in $\Omega:=[0,1]\times[0,1]$ with zero Dirichlet boundary conditions. To generate 
solutions with different smoothness   
we use two different right hand sides, namely, $f_1$ and $f_2$ (see, Example 1 and 2 below). 
The vectors $ \cI_h f_1 $ and $\cI_h f_2$ representing the data for the linear system 
\eqref{fda} are obtained by evaluating the functions $f_1$ and $f_2$ at the 
mesh points taken in lexicographical order. 
At the point of discontinuity, the values are taken to be zero.
\paragraph{\it Example 1} The right hand side  $f_1(x,y)$,  used also in the numerical tests in \cite{BP15,HLMMV18}, 
is piece-wise constant function ({\em CheckerBoard}), which has 
jump discontinuities along 
the lines $x=0.5$ and $y=0.5$
\begin{equation}\label{eq:rhs}
f_1(x,y)=\left\{
                    \begin{array}{rl} 
                    1,& \text{if } (x-0.5)(y-0.5)>0,\\
                    -1,& \text{if } (x-0.5)(y-0.5)<0.
                    \end{array}\right.\quad
\end{equation}
As $f_1$ is not continuous, $\cI_h f_1$ is not well defined.  Instead,
we set $\tilfh(\bx_j)=0 $ at points of discontinuity and
$\tilfh(\bx_j)=f(\bx_j)$ otherwise.  Here $ \{\bx_j\}$ are the interior
nodes of the finite element mesh.
A discrete reference solution of $\tiluh=\wcalAt^{-\alpha}\cI_h f_1$, where $\calAt$ is as in \eqref{FD-matrix}, 
has been computed using FFT techniques on a uniform mesh with mesh-size $h=2^{-15}$. 

\paragraph{\it Example 2} Now we consider smooth data $f_2(x,y) = \sin(2\pi x) \sin(2\pi y)$. Since 
$f_2$ is an eigenfunction of both the Laplace and the discrete Laplace operators,  the exact discrete solution 
on the uniform mesh with mesh-size $h$ is
$$
\tiluh =\wcalAt^{-\alpha}\cI_h f_2=\left(\frac{8\sin^2(\pi h)}{h^2}\right)^{-\alpha}\cI_h f_2.
$$
Together with the two BURA-related solvers (BURA and P-BURA), 
we apply also the method, proposed by Bonito and Pasciak in \cite{BP15} that
incorporates an exponentially convergent quadrature scheme for 
approximation of integral representation of the solution \eqref{bal}
\begin{equation*}
t^{-\alpha} \approx Q_\alpha(t):=\frac{2k'\sin(\pi\alpha)}{\pi}\sum_{\ell=-m}^M 
\frac{e^{2(\alpha-1)\ell k'}}{t+e^{-2\ell k'}},\qquad t\in(0,\infty), 
\end{equation*}
where $m=\lceil (1-\alpha)k\rceil$, $M=\lceil \alpha k\rceil$, $k'=\pi/(2\sqrt{\alpha(1-\alpha)k})$. 
Note that $Q_\alpha$ is in the class of rational functions  $\rat {k+1}$ or $\rat {k+2}$. 
In particular, for $k=7$, $Q_\alpha \in \rat 9$ when $\alpha=\{0.25,0.5,0.75\}$. 
The approximate solution is of the form
\begin{equation}\label{eq:Q-method}
\tilde u_{h,Q}:=\frac{2k'\sin(\pi\alpha)}{\pi}\sum_{\ell=-m}^M 
e^{2(\alpha-1)\ell k'}\left(\calAt+e^{-2\ell k'} \calIt \right)^{-1}f_h . 
\end{equation}
The parameter $k'>0$ controls the accuracy of $u_{h,Q}$ and the number of linear systems to be solved. 
For example, $k'=1/3$ gives rise to 120 systems for $\alpha=\{0.25,0.75\}$ and $91$ systems for $\alpha=0.5$ 
guaranteeing  $\| \tilde u_{h,Q}- \tilde u_h \|_{\ell_2} \approx 
10^{-7} \| f_h \|_{\ell_2}$. We will refer to such a parameter choice as {\em the $k^\prime$-Q-method}. 
On the other hand, taking $k=7$ we need to solve nine linear systems in order to derive $u_{h,Q}$, and this 
will be called {\em the Q-method}.
Although the theoretical foundation of the Q-method is quite different from the 
one of the BURA-related methods, both of the approaches are computationally very similar and this 
allows us to perform a meaningful comparison analysis.

 \subsection{Numerical tests for uniform mesh} \label{ss:uniform} 
 Now we analyze the computational results from the efficiency point of view. 
We fix the number $k$ in such a way that  the three methods, BURA, P-BURA and Q-method, 
require $9$ systems of the type  $(\wcalIt - d_i  \wcalAt )  \tilwh= \tilvh$ to be solved. 
This means that all three methods  need almost the same amount of computational work. 
For comparison, we also give the results of the $k^\prime$-Q-method  that has the best 
accuracy, but requires about 10 - 15 times more computational work.

Tables \ref{tab:2D results 1} and  \ref{tab:2D results 2} 
present the computational results for the four  solvers discussed
above for three values of $\alpha=0.25, 0.5, 0.75$ requiring a number of solves as
discussed above. Together with the $\ell_2$-norm of the error we also provide (just for
comparison purposes) the error measured on the maximum norm $\|\cdot \|_{\ell_\infty}$.

The first general comment is that all methods work according to the developed theory. In terms of 
efficiency the P-BURA method  seems to be the best. In agreement with the theory, the accuracy of the 
method does not depend on the mesh size $h$. Also, in agreement with the approximation error 
reported in Table \ref{t:error} its error decreases when $\alpha$ increases. But even in the worst
approximation, the case when $\alpha=0.25$, P-BURA produces a reasonable error in the range of $10^{-4}$
when using only 9 system solves.  Moreover, for a fixed mesh of  medium mesh-size (say, $h=10^{-8} - 10^{-9}$) 
and $\alpha=0.25$ P-BURA is as accurate as 
BURA method and outperforming BURA and Q-method on all meshes for $\alpha=0.5$ and $\alpha=0.75$ 
on both Problem 1 and Problem 2.
In contrast, the $k^\prime$-Q-method has the same accuracy,
but needs 120 system solves. 

Second, we note that from the first row of Table \ref{t:error} we see that the BURA approximation for 
$t^{1-\alpha}=t^{0.75}$  has good accuracy for relatively small values of $k$.
We see that for $k=8$ the error ranges from $4.4\times 10^{-4}$ for $\alpha=0.75$ to $1.2 \times 10^{-6}$ for $\alpha=0.25$.
However, the computational results on Tables  \ref{tab:2D results 1} and  \ref{tab:2D results 2} show that
the factor $\kappa(\wcalAt)^{1-\alpha} $  in the error bound for BURA method is polluting the 
approximate solution and reducing the accuracy. This pollution is especially visible in the computational results
for $\alpha=0.25$. In this case $\kappa(\wcalAt)^{1-\alpha} =\kappa(\wcalAt)^{3/4}=O( h^{-3/2}) $
and  every time one halves the mesh-size the error is increased by a factor of $ 2^{3/2} \approx 3.8$.
This pollution is less visible for $\alpha=0.75$ since the factor is $2^{1/2} \approx 1.4$.
Regardless of this pollution, the BURA method with 8 system solves is, in general, more accurate 
than the Q-method for all three values of $\alpha$, when using the same number of system solves.

\begin{table}[h!]
\centering
\caption{Relative errors of the approximate solution for $\alpha=0.25,0.5,0.75$
obtained by four different methods on various uniform meshes. Each of the methods
BURA, P-BURA, and  Q-method uses $9$ linear systems, while the  $k^\prime$-Q-method, 
$k^\prime=1/3$, uses $120$ linear system solves for $\alpha=0.25,0.75$ and  $91$ solves
 for $\alpha=0.5$. The reference solution is the discrete solution
on a mesh with step-size $h=2^{-15}$ computed via FFT.}\label{tab:2D results 1}
 \begin{tabular}{|c|c|cc|cc|cc|cc|}
\hline
\multirow{3}{*}{$\alpha$} & \multirow{3}{*}{$h$} & \multicolumn{8}{|c|}{{\it Example 1}, CheckerBoard right-hand-side}\\ \cline{3-10}
& & \multicolumn{2}{|c|}{BURA} & \multicolumn{2}{|c|}{P-BURA} & \multicolumn{2}{|c|}{Q-method} & \multicolumn{2}{|c|}{$k'$-Q-method}\\
& &  $\ell_2$ & $\ell_\infty$ &  $\ell_2$ & $\ell_\infty$ &  $\ell_2$ & $\ell_\infty$ &  $\ell_2$ & $\ell_\infty$ \\ \hline
\multirow{5}{*}{$0.25$} & 
  $2^{-8}$  & 2.929e-4 & 2.612e-3 & 3.255e-4 & 2.550e-3 & 1.045e-2 & 1.288e-2 & 2.772e-4 & 2.612e-3 \\
& $2^{-9}$  & 1.747e-4 & 1.847e-3 & 2.292e-4 & 1.875e-3 & 1.040e-2 & 1.207e-2 & 1.371e-4 & 1.847e-3 \\
& $2^{-10}$ & 8.217e-4 & 1.829e-3 & 2.029e-4 & 1.339e-3 & 1.039e-2 & 1.152e-2 & 6.815e-5 & 1.305e-3 \\
& $2^{-11}$ & 5.077e-3 & 1.094e-2 & 1.939e-4 & 8.219e-4 & 1.038e-2 & 1.097e-2 & 3.388e-5 & 9.196e-4 \\
& $2^{-12}$ & 1.129e-2 & 2.610e-2 & 1.922e-4 & 7.451e-4 & 1.038e-2 & 1.069e-2 & 1.671e-5 & 6.413e-4 \\\hline
\multirow{5}{*}{$0.50$} & 
  $2^{-8}$  & 9.688e-5 & 1.900e-4 & 2.212e-5 & 1.849e-4 & 2.847e-3 & 2.910e-3 & 2.331e-5 & 1.821e-4 \\
& $2^{-9}$  & 2.337e-4 & 4.995e-4 & 1.013e-5 & 8.787e-5 & 2.835e-3 & 2.904e-3 & 8.058e-6 & 9.110e-5 \\
& $2^{-10}$ & 3.828e-4 & 8.616e-4 & 8.304e-6 & 4.742e-5 & 2.830e-3 & 2.902e-3 & 2.840e-6 & 4.559e-5 \\
& $2^{-11}$ & 2.413e-4 & 6.274e-4 & 8.263e-6 & 2.433e-5 & 2.829e-3 & 2.902e-3 & 1.033e-6 & 2.280e-5 \\
& $2^{-12}$ & 1.424e-3 & 2.814e-3 & 8.291e-6 & 1.909e-5 & 2.828e-3 & 2.902e-3 & 4.118e-7 & 1.132e-5 \\\hline  
\multirow{5}{*}{$0.75$} & 
  $2^{-8}$  & 1.219e-4 & 2.741e-4 & 2.443e-6 & 9.168e-6 & 1.507e-3 & 1.825e-3 & 2.561e-6 & 9.103e-6 \\
& $2^{-9}$  & 1.761e-4 & 3.976e-4 & 6.110e-7 & 3.110e-6 & 1.502e-3 & 1.824e-3 & 7.118e-7 & 3.263e-6 \\
& $2^{-10}$ & 2.172e-4 & 4.958e-4 & 1.884e-7 & 1.037e-6 & 1.501e-3 & 1.823e-3 & 2.355e-7 & 1.198e-6 \\
& $2^{-11}$ & 1.401e-4 & 3.478e-4 & 1.500e-7 & 6.592e-7 & 1.500e-3 & 1.823e-3 & 1.138e-7 & 4.677e-7 \\
& $2^{-12}$ & 1.803e-4 & 3.264e-4 & 1.547e-7 & 4.574e-7 & 1.499e-3 & 1.823e-3 & 8.334e-8 & 2.079e-7 \\\hline 
\end{tabular}
\end{table}

\begin{table}[h!]
\centering
\caption{Relative errors of the approximate solution for $\alpha=0.25,0.5,0.75$
obtained by four different methods on various uniform meshes. Each of the first three solvers 
incorporates $9$ linear systems, while the last solver incorporates $120$  for $\alpha=0.25,0.75$ and
$91$ system solves  for $\alpha=0.5$. On each level we have computed the exact Galerkin solution 
$\tiluh=\wcalAt^{-\alpha}\cI_h f_2=\lambda_2^{-\alpha} \cI_h \Psi_2$.
}
\label{tab:2D results 2}
 \begin{tabular}{|c|c|cc|cc|cc|cc|}
\hline
\multirow{3}{*}{$\alpha$} & \multirow{3}{*}{$h$} & \multicolumn{8}{|c|}{{\it Example 2}, $f_2(x,y) = \sin(2\pi x) \sin(2\pi y)$}\\ \cline{3-10}
& & \multicolumn{2}{|c|}{BURA} & \multicolumn{2}{|c|}{P-BURA} & \multicolumn{2}{|c|}{Q-method} & \multicolumn{2}{|c|}{$k'$-Q-method}\\
& &  $\ell_2$ & $\ell_\infty$ &  $\ell_2$ & $\ell_\infty$ &  $\ell_2$ & $\ell_\infty$ &  $\ell_2$ & $\ell_\infty$ \\ \hline
\multirow{5}{*}{$0.25$} & 
  $2^{-8}$  & 4.615e-5 & 9.193e-5 & 1.086e-4 & 2.162e-4 & 5.223e-3 & 1.040e-2 & 1.955e-6 & 3.894e-6 \\
& $2^{-9}$  & 6.035e-5 & 1.205e-4 & 1.067e-4 & 2.131e-4 & 5.214e-3 & 1.041e-2 & 3.688e-7 & 7.362e-7 \\
& $2^{-10}$ & 4.993e-4 & 9.976e-4 & 1.062e-4 & 2.123e-4 & 5.209e-3 & 1.041e-2 & 2.663e-8 & 5.321e-8 \\
& $2^{-11}$ & 3.122e-3 & 6.240e-3 & 1.061e-4 & 2.121e-4 & 5.207e-3 & 1.041e-2 & 1.253e-7 & 2.506e-7 \\
& $2^{-12}$ & 6.904e-3 & 1.380e-2 & 1.060e-4 & 2.120e-4 & 5.206e-3 & 1.041e-2 & 1.500e-7 & 2.999e-7 \\\hline   
\multirow{5}{*}{$0.50$} & 
  $2^{-8}$  & 6.388e-5 & 1.273e-4 & 5.703e-6 & 1.136e-5 & 1.428e-3 & 2.845e-4 & 1.329e-6 & 2.648e-6 \\
& $2^{-9}$  & 1.426e-4 & 2.846e-4 & 4.630e-6 & 9.243e-6 & 1.426e-3 & 2.847e-3 & 2.650e-7 & 5.290e-7 \\
& $2^{-10}$ & 2.360e-4 & 4.715e-4 & 4.361e-6 & 8.713e-6 & 1.425e-3 & 2.847e-3 & 3.273e-10& 6.539e-10\\
& $2^{-11}$ & 1.469e-4 & 2.936e-4 & 4.292e-6 & 8.580e-6 & 1.424e-3 & 2.847e-3 & 6.656e-8 & 1.331e-7 \\
& $2^{-12}$ & 8.771e-4 & 1.754e-3 & 4.275e-6 & 8.547e-6 & 1.424e-3 & 2.848e-3 & 8.310e-8 & 1.662e-7 \\\hline
\multirow{5}{*}{$0.75$} & 
  $2^{-8}$  & 7.299e-5 & 1.454e-4 & 7.564e-7 & 1.507e-6 & 8.331e-4 & 1.660e-3 & 6.733e-7 & 1.341e-6 \\
& $2^{-9}$  & 1.086e-4 & 2.168e-4 & 2.208e-7 & 4.408e-7 & 8.320e-4 & 1.661e-3 & 1.379e-7 & 2.753e-7 \\
& $2^{-10}$ & 1.332e-4 & 2.662e-4 & 8.724e-8 & 1.743e-7 & 8.314e-4 & 1.661e-3 & 4.375e-9 & 8.742e-9 \\
& $2^{-11}$ & 8.546e-5 & 1.708e-4 & 5.387e-8 & 1.077e-7 & 8.310e-4 & 1.661e-3 & 2.896e-8 & 5.789e-8 \\
& $2^{-12}$ & 1.110e-4 & 2.219e-4 & 4.553e-8 & 9.103e-8 & 8.308e-4 & 1.661e-3 & 3.728e-8 & 7.454e-8 \\\hline
\end{tabular}
\end{table}

\subsection{Numerical tests on locally refined mesh} 
\label{ss:non-uni}
Since the P-BURA error estimate \eqref{pbest-fd} is independent of
the condition number of the discretization matrix $\wcalAt$, one can also 
apply local refinement techniques for efficiently capturing the solution 
behavior around possible singularities of the solution $ u $. In this section we 
illustrate the advantages of such an approach, considering one-dimensional example for 
which the exact continuous solution of the fractional diffusion problem is 
analytically known. We always perform geometric dyadic refinement around the 
singularities and apply the P-BURA method as a solver.

The eigenfunctions and the eigenvalues of the 1-dimensional  problem \eqref{strong} are
\begin{equation}\label{eq:1DLaplace}
\psi_i(x)=\sqrt{2}\sin(\pi ix);\qquad \mu_i=i^2\pi^2,\qquad \forall i\in\NN.                                                                                                               
\end{equation}
Note that with respect to the standard $L^2[0,1]$ inner product, 
we have $(\psi_i,\psi_j)=\delta_{ij}$ for all $i,j\in\NN$. Therefore, 
for any right-hand side function $f$ on $(0,1)$ we can explicitly 
compute the solution of the continuous fractional diffusion problem 
with homogeneous boundary conditions
\begin{equation}\label{eq:1Dexpand}
u(x) :=\calA^{-\alpha}f=\sum_{i=1}^\infty \mu^{-\alpha}_i(f,\psi_i)\psi_i(x). 
\end{equation}
Furthermore, we have $\lambda_1=\pi^2 > 1$ so for all meaningful grids on $[0,1]$ 
(e.g., coming from finite element or finite difference discretization) the first 
eigenvalue $\lambda_1$ of the corresponding discrete operator $\calAt$ satisfies 
$\lambda_1\ge\mu_1>1$. Thus the spectrum of $\calAt$ is always in $[1,\infty)$ 
and we do not need to normalize the matrix in order to apply the P-BURA solver.

Now we take a smooth function, namely, $f(x) =1$, but the solution of 
\eqref{frac-eq} will exhibit boundary layers near the end-points $x=0$ and $x=1$.
Those layers are steeper as 
$\alpha\to 0$ and in order for the numerical solver to correctly capture them, 
we need very fine mesh near the boundary, especially for small $\alpha$. 
Thus, we consider the following class of locally refined meshes: Firstly, we 
start with a uniform mesh of size $h_0$. Then, at each refinement step we take 
the first and the last segments (those that have a boundary point at $0$, 
respectively a boundary point at $1$) and subdivide them on $p$ equal parts, 
introducing $p-1$ new mesh nodes per segment. 

Direct computations give rise to 
$$
(f,\psi_i)=\int_0^1\sqrt{2}sin(\pi ix)dx=\left\{\begin{array}{rl} 0, 
        & \text{if $i$ is even};\\\frac{2\sqrt{2}}{i\pi}, & \text{if $i$ is odd}.\end{array}\right.
$$
and due to \eqref{eq:1Dexpand} we have the explicit representation of the exact solution
\begin{equation}\label{sol:1Done}
u(x)=\frac{2\sqrt{2}}{\pi^{1+2\alpha}}\sum_{i=0}^\infty \frac{\psi_{2i+1}(x)}{\left(2i+1\right)^{1+2\alpha}} .
\end{equation}
As a reference solution, we consider the truncated series representation \eqref{sol:1Done} by taking 
the first $10^4$ terms. Of course, this is an approximation to the exact solution. However, the error 
of such approximations for $ \alpha =0.25, 0.5, 0.75$  are all less than $ 1.270\text{e-5}$,   
$4.136\text{e-8}$, $ 1.429\text{e-10}$, respectively.  Since these are all substantially below the 
approximation error of BURA (see, Table \ref{t:error}), we use them as substitute of the exact solution.

The accuracy of this reference solution is higher than the accuracy of the P-BURA method.
We study the relative $L_2$-error
$
{\|w_{h}-u_h\|}/{\|u\|},
$
where $w_h$ is a piece-wise linear function that interpolates the P-BURA solution on the locally refined mesh, 
while $u_h$ is a piece-wise linear function that samples $u$ on the uniform mesh with $h=2^{-18}$. 

The numerical results are summarized on Table~\ref{tab:1Dconst}. We have considered only dyadic mesh refinement (i.e., $p=2$). 
As expected, the smaller the $\alpha$ is, the steeper the boundary layers are, 
thus the bigger the benefit of the local refinement is. For example, for $\alpha=0.25$ we observe 
that starting with a uniform mesh of size $h_0=2^{-6}$ and performing 9 
additional local refinement steps, we end up 
with a numerical solution that is as accurate as the numerical solution on a uniform mesh of size $h_0=2^{-10}$. 
On the other hand, the first mesh consists of $81$ nodes, white the second one - of $1023$. Moreover, as $h_0$ 
increases the order of the error is the same as the order of the 9-BURA accuracy $E_{0.25,9}=2.654$e-4 
(see Table~\ref{t:error}), meaning that the geometrical 6-step adaptive refinement with $h_0=2^{-9}$ leads 
to almost optimal results at the numerical cost of solving nine 
tridiagonal linear systems of size $523$.  The latter is further illustrated on Fig.~\ref{fig:2D}. 
There, using that the CheckerBoard function on $[0,1]\times[0,1]$ can be split into four squared 
sub-domains, such that on each of them we solve a tensor product of two 1D problems like {\em Example 3}, 
together with the linearity of the fractional Laplace operator, we numerically compute solution of 
{\em Example 1} for $\alpha=0.25$ on a $(2\cdot523+1)\times(2\cdot523+1)=1047\times1047$ mesh, 
which is locally adapted along the boundary of the domain and the lines of discontinuity of the right-hand-side. 
On the left, we plot the computed numerical solution. In the middle we show the mesh refined around the 
central point $(0.5,0.5)$. On the right we plot the point-wise error between the numerical solution and the 
true exact solution, sampled at a uniform grid of size $h=2^{-16}$ in the same region of interest. 
Note that the region captures the boundary layers along $x=0.5$ and $y=0.5$ and the 
point-wise error is less than $2$e-4 overall.

\begin{table}[h!]
\centering
\caption{Computing $\| w_h-u_h\|/\|u\|$ for $f(x)=1$ and $\alpha=\{0.25,0.5,0.75\}$ 
on various uniform and locally refined meshes. 9-P-BURA is used as solver and all meshes have 
been iteratively refined until the smallest mesh segment is of size $2^{-15}$. 
Dyadic refinement ($p=2$) has been applied.} 
\label{tab:1Dconst}
\begin{tabular}{|c|c|c|c|c|c|c|}
\hline
 & Ref. level & $h_0=2^{-6}$ & $h_0=2^{-7}$ & $h_0=2^{-8}$ & $h_0=2^{-9}$ & $h_0=2^{-10}$ \\\hline 
\multirow{2}{*}{$\alpha=0.25$} & 
  $0$  & 1.813e-2 & 9.071e-3 & 4.544e-3 & 2.287e-3 & 1.179e-3\\
& last & 1.294e-3 & 6.931e-4 & 4.446e-4 & 3.541e-4 & 3.301e-4\\\hline   
\multirow{2}{*}{$\alpha=0.50$} & 
  $0$  & 2.705e-3 & 9.547e-4 & 3.380e-4 & 1.233e-4 & 5.498e-5\\
& last & 7.620e-4 & 2.648e-4 & 9.487e-5 & 4.563e-5 & 3.734e-5\\\hline
\multirow{2}{*}{$\alpha=0.75$} & 
  $0$  & 6.415e-4 & 1.713e-4 & 4.882e-5 & 2.158e-5 & 1.795e-5\\
& last & 4.712e-4 & 1.321e-4 & 4.054e-5 & 2.056e-5 & 1.792e-5\\\hline
$\#$ mesh & 
  $0$  & 63 & 127 & 255 & 511 & 1023\\
nodes & last & 81 & 143 & 269 & 523 & 1033\\\hline
\end{tabular}
\end{table}

\begin{figure}[b!]
\begin{tabular}{ccc}
\includegraphics[width=0.33\textwidth]{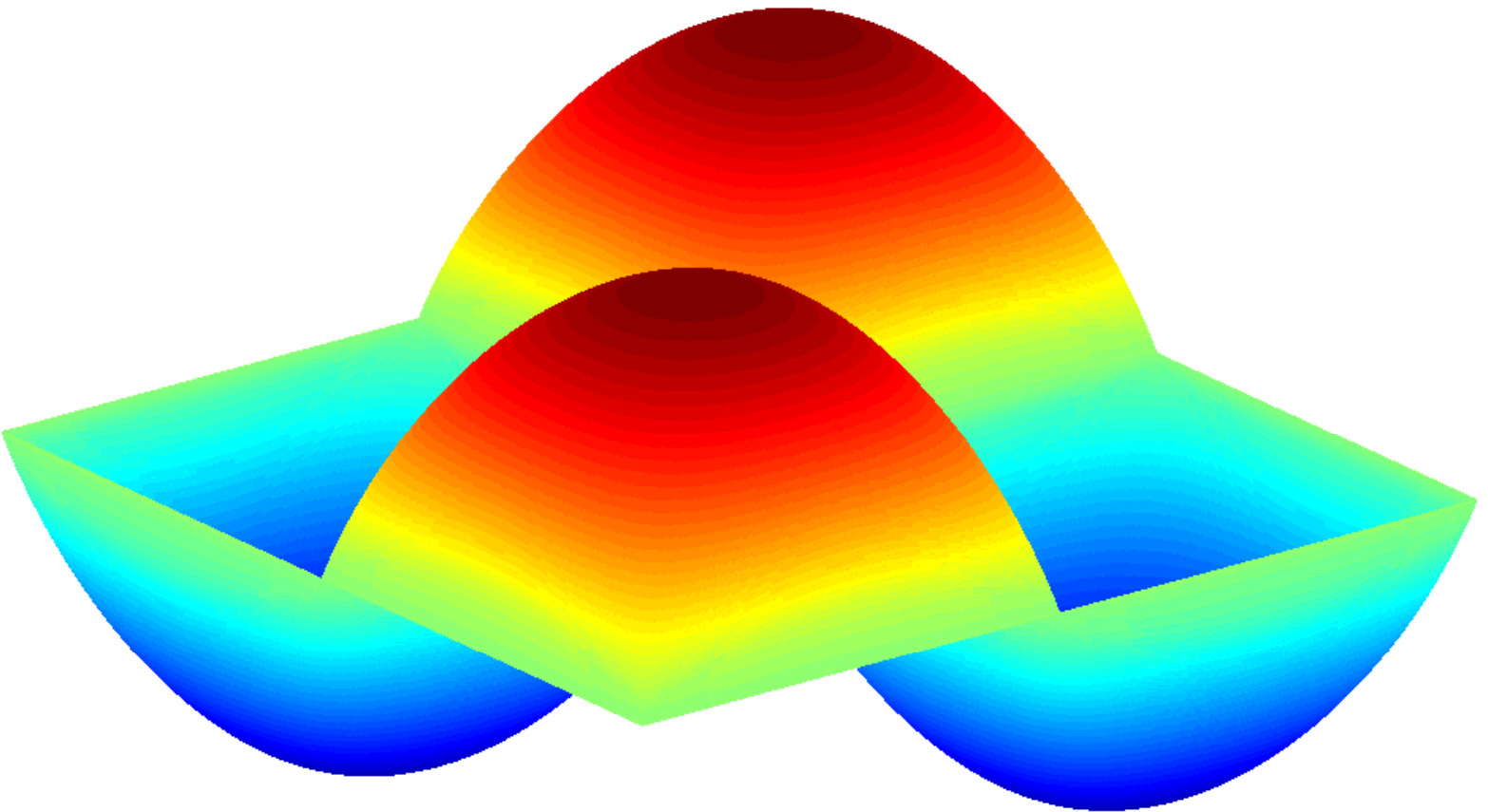} &
\includegraphics[width=0.33\textwidth]{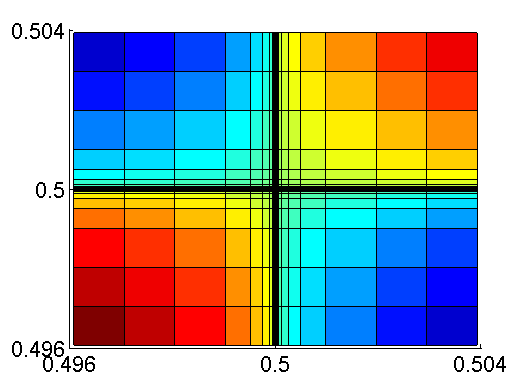} &
\includegraphics[width=0.33\textwidth]{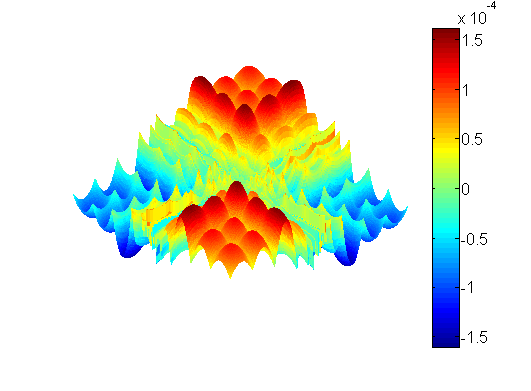}
\end{tabular}
\caption{Computing the CheckerBoard problem for $\alpha=0.25$ on a locally adaptive 2-dimensional mesh. 
Left: Tensor-product numerical solution. Center: Geometrical refinement along discontinuities: the mesh 
within the domain $\Omega^\prime=[0.496,0.504] \times [0.496,0.504]$. 
Right: Plot of the  error $(u_P-u_h)(x) $ over $\Omega^\prime$ for $h=2^{-16}$.}
\label{fig:2D}
\end{figure}

\begin{table}[h!]
\centering
\caption{Computing $\| w_h-u_h\| /\|u \|$ for $f(x)=\delta_{1/2}(x)$ 
 and $\alpha=\{0.5,0.75\}$ on various uniform and locally refined meshes. 
 P-BURA is used as solver and meshes have been iteratively refined until the 
 smallest mesh segment is of size $2^{-16}$, respectively $2^{-13}$, for 
 $\alpha=0.5$, $\alpha=0.75$. Dyadic refinement ($p=2$) has been applied.} 
\label{tab:1Dsource}
 \begin{tabular}{|c|c|c|c|c|c|c|}
\hline
 & Ref. level & $h_0=2^{-6}$ & $h_0=2^{-7}$ & $h_0=2^{-8}$ & $h_0=2^{-9}$ & $h_0=2^{-10}$ \\\hline 
\multirow{2}{*}{$\alpha=0.50$} & 
  $0$  & 9.218e-2 & 6.510e-2 & 4.610e-2 & 3.243e-2 & 2.309e-2\\
& last & 1.026e-2 & 7.975e-3 & 6.559e-3 & 5.690e-3 & 5.267e-3\\\hline
\multirow{2}{*}{$\alpha=0.75$} & 
  $0$  & 5.776e-3 & 2.882e-3 & 1.452e-3 & 7.138e-4 & 3.610e-4\\
& last & 1.456e-3 & 7.345e-4 & 3.826e-4 & 2.084e-4 & 1.423e-4\\\hline
$\#$ mesh & 
  $0$  & 63 & 127 & 255 & 511 & 1023\\
nodes & last & 83/77 & 145/139 & 269/263 & 525/519 & 1035/1029\\\hline
\end{tabular}
\end{table}

\paragraph{\it Example 4}
We consider the problem 
$
- u^{\prime\prime}=\delta_{1/2}(x)$ for $ 0<x<1, \ u(0)=u(1)=0,$
where the Dirac delta function $ \delta_{1/2}(x)$ is centered at $0.5$.
The weak formulation of this problem is: find $u\in H^1_0(0,1)$
satisfying
$$  
\int_0^1 u^\prime \phi^\prime \, dx = \int_0^1 \delta_{1/2}\phi\, dx =\phi(1/2), \Forall \phi\in H^1_0(0,1).
$$
It follows that 
$$\delta_{1/2} = \sum_{i=1}^\infty \psi_i(1/2) \psi_i(x)= \sqrt2 
\sum_{i=0}^\infty (-1)^i \psi_{2i+1}(x).$$
Hence, 
\begin{equation}\label{sol:1Dsource}
u(x)=  \sqrt{2}\pi^{-2\alpha}  \sum_{i=0}^\infty \frac{(-1)^i}{(2i+1)^{2\alpha}}\psi_{2i+1}(x) \quad \mbox{and} \quad 
\|u\|=\sqrt{2}\pi^{-2\alpha}\bigg(\sum_{i=0}^{\infty}(2i+1)^{-4\alpha}\bigg)^{1/2}. 
\end{equation}
Since for $\alpha \le 0.25$ the series $ \sum_{i=0}^{\infty}{(2i+1)^{-4\alpha}}$ does not converge,
$u(x) \in L^2(0,1)$ only for $\alpha>0.25$. 
Therefore, here we considered the cases $\alpha=0.5$ and $\alpha=0.75$ only. Since the singularity 
is at $0.5$, we start with a uniform mesh of size $h_0$, at each refinement step we take 
the two central segments (those that have a boundary point at $1/2$) and divide them in halves. 
Once the mesh is fixed, we take $\tilfh$ to be zero everywhere but in the middle, where the 
value is set to $h^{-1}_\ast$ -- the size of the segments, attached to the midpoint. 
As before, we  truncate the infinite series  \eqref{sol:1Dsource} to
produce approximation to the  solution $u$ which error is far below the error due to BURA.
In  Table~\ref{tab:1Dsource} we summarize our study of the  relative $L_2$-error: 
$
{\| w_h-u_h\|}/{\|u \|}.
$

\section{Error estimates for the semi-discrete finite element approximations}
\label{sec:lumped}

In this section, we consider finite element approximations to the
solution of the fractional problem $\calA^{\alpha} u= f$ (or equivalently $u=\calA^{-\alpha} f$).   For
simplicity, we only consider the case when the solution operator $T$
satisfies full elliptic regularity, i.e. for $f\in L^2(\Omega)$, the
solution $v=Tf$ of \eqref{weak} is in $H^2(\Omega)\cup H^1_0(\Omega)$
and satisfies
\beq
\| v \|_{H^2(\Omega)} \le c \|f\|.
\label{fullreg}
\eeq
The assumption of full regularity greatly simplifies the semi-discrete
error analysis.  Further, to avoid proliferation of various constant related to the 
maximum and minimum values of $a(x)$ in this Section we shall use the
norm generate by the bilinear form $a(\cdot, \cdot)$ which is equivalent to $H^1$:
\beq\label{a-norm}
\| u \|_a = a(u,u)^\frac12 \ \Forall u \in H^1_0(\Omega).
\eeq


\subsection{Consistent mass finite element method}
\label{ss:consistmass}
Now we consider the finite elements  method  \eqref{mat-FEM}  with consistent 
mass computation. \\

\paragraph{\it Approximation properties of the method}
To study the approximation properties of $u_h$, we follow the technologies 
developed in \cite {fujitaSuzuki, Ushijima1979} and include the 
proof for completeness.

\begin{theorem} \label{FS-thm} (Fujita-Suzuki, \cite[Theorem 5.2, p.~806]{fujitaSuzuki}) 
Suppose that \eqref{fullreg} holds.
  Then for $f\in L^2(\Omega)$, 
$$\|\calA^{-\alpha} f -\calAt^{-\alpha} \pi_h f \| \le C h^{2\alpha} \|f\|$$
with $C$ not depending on $h$.  Here $\calAt$ denotes the finite
element operator without lumping, i.e., that appearing in \eqref{fea}.
\end{theorem}

\begin{proof}  Using the Balikrishnan formula \eqref{bal} gives
\beq
\calA^{-\alpha} f -\calAt^{-\alpha} \pi_h f =\frac {\sin(\pi \alpha)} \pi 
\int_0^\infty \mu^{-\alpha} [ w_\mu-w_{\mu,h}]
\, d\mu.
\label{bald}
\eeq
where 
$$
w_\mu:=(\mu \calI +\calA)^{-1}f  \hbox{ and } w_{\mu,h} =(\mu \calI +\calAt)^{-1}\pi_h f.
$$
We clearly have that $w_\mu\in H^1_0(\Omega)$ is the unique solution of 
\beq
\mu (w_\mu,\theta)+a(w_\mu,\theta)=(f,\theta)\qquad \Forall \theta\in
H^1_0(\Omega)
\label{wt}
\eeq
and  $w_{\mu,h} \in V_h$ 
 is the unique solution of 
\beq
\mu (w_{\mu,h},\theta)+a(w_{\mu,h},\theta)=(f,\theta)\qquad \Forall \theta\in
V_h.
\label{wth}
\eeq
Now, taking $\theta=w_\mu$ in \eqref{wt} and applying the Schwarz
inequality to the right hand side implies   that
$$\|w_\mu\|\le \mu^{-1} \|f\|$$
and subsequently
\beq
\|w_\mu\|_a \le \mu^{-1/2} \|f\|.
\label{onen}
\eeq
Here $\|\cdot\|_a:=a(\cdot,\cdot)^{1/2}$ denotes the $a$-norm on
$H^1_0(\Omega)$.  Let $e_\mu=w_\mu-w_{\mu,h}$.  Using \eqref{onen},
Galerkin orthogonality and standard error estimates for finite element
approximation 
gives, for all $\chi\in V_h$,
$$\bal
\mu \|e_\mu\|^2 + a(e_\mu,e_\mu) &= \mu (e_\mu,w_\mu-\chi) + a(e_\mu,w_\mu-\chi) \\
&\le c h \big [\mu^{1/2} \|e_\mu\| + \|e_\mu\|_a\big ] \|f\|.
\eal$$
A simple application of the arithmetic-geometric mean inequality implies 
\beq  \|e_\mu\|_\mu:= (\mu \|e_\mu\|^2 + a(e_\mu,e_\mu))^{1/2}\le c h \|f\|.
\label{etsum}
\eeq
We apply finite element duality defining $z\in H^1_0(\Omega)$ to be the
solution of
$$\mu (\theta,z)+a(\theta,z)=(\theta_,e_\mu),\Forall \theta\in
H^1_0(\Omega),$$
so that  Galerkin orthogonality implies 
$$\bal \|e_\mu\|^2 &=  \mu (e_\mu,z) + a(e_\mu,z) 
= \mu (e_\mu,z-\chi) + a(e_\mu,z-\chi) ,\Forall \chi\in V_h.
\eal$$
Now the Schwarz inequality, \eqref{etsum} 
 and arguments leading to \eqref{etsum} applied to $e^z_\mu:=z-\chi$ gives
$$\|e_\mu\|^2\le \|e_\mu\|_\mu \|e_\mu^z\|_\mu\le
c h^2 \|f\|  \|e_\mu\|$$
and so
\beq
\|e_\mu\|\le ch^2 \|f\|.
\label{et1}
\eeq

Now, taking $\theta=w_\mu$ in \eqref{wt} and $\theta = w_{\mu,h}$ in
\eqref{wth} gives
$$ \|w_\mu\|\le \mu^{-1} \|f\| \hbox{ and }  \|w_{\mu,h}\|\le \mu^{-1}
\|f\| $$
and so
\beq
\|e_\mu\|\le 2\mu^{-1} \|f\|.
\label{et2}
\eeq
Using the above estimates in \eqref{bald} gives
$$\bal 
\|\calA^{-\alpha} f &-\calAt^{-\alpha} \pi_h f\| \le \frac {\sin(\pi \alpha)} \pi 
\bigg[ \int_0^{h^{-2}} \mu^{-\alpha} \| e_\mu \| \, d\mu
+\int_{h^{-2}}^\infty 
\mu^{-\alpha}\| e_\mu
  \| \, d\mu \bigg]\\
&\le  
 c\|f\|\bigg[ \int_0^{h^{-2}}\mu^{-\alpha}  h^2   
\, d\mu  +\int_{h^{-2}}^\infty \mu^{-1-\alpha}  \, d\mu \bigg] \le c h^{2\alpha} \|f\|.
\eal
$$
This completes the proof of the theorem.
\end{proof}

\paragraph{\it Positivity of the approximate solution}
Note that diffusion problem \eqref{strong} is nonnegative if $f$ is nonnegative. 
This property is retained by the finite difference approximation of the problem.
In the case of finite element method  \eqref{classic-FEM} we see that 
if $\calSt$ is an $M$-matrix, then for $f \ge 0$ and consequently 
have ${\widetilde F}  \ge0$  and 
$ \tiluh \ge 0$, i.e. the finite element method 
preserves positivity.

Next, we ask the question whether the finite element solution  \eqref{mat-FEM} 
of the sub-diffusion problem \eqref{frac-eq} retains this property. 
Obviously, the solution \eqref{mat-FEM} can be expressed  by \eqref{bal}
$$
\tiluh 
= \wcalAt^{-\alpha} \wcalMt^{-1} \tilF = \frac {\sin(\pi \alpha)} \pi 
\int_0^\infty \mu^{-\alpha} (\mu \wcalMt +\calSt)^{-1}\, d\mu \ \tilF
$$
Since $\wcalMt$ is not an $M$-matrix, from this representation one can conjecture that 
even if $\calSt$ is an $M$-matrix, 
$ \wcalMt \wcalAt^{\alpha} $ could fail to be an $M$-matrix and the finite element scheme
may be not positivity preserving. For this, 
we have made some 
direct computations of the entries of the matrix $\wcalMt \wcalAt^{\alpha} $ for the case of one-dimensional
problems for various $\alpha$ and  step-sizes $h=1/(N+1)$. 
In this case the  $N \times N$ matrix $\calSt$ is defined by \eqref{FD-matrix-1D} (with $h_i=h$) and 
 $\wcalMt=\frac{h}{6}diag(1, 4, 1) $ (a tridiagonal matrix). 

In Table \ref{t:positivity} we present the following information regarding 
 the matrix $\wcalMt \wcalAt^{\alpha} $ for various 
$\alpha$ and step-size $h$: in columns MrowS, we report the 
maxim row-sum and in columns MoffD, we report the maximum of all 
off-diagonal elements. We see that all row sums are positive. 
It is clear that if the maximal off-diagonal element 
is positive, then the matrix is NOT an $M$-matrix and therefore we cannot conclude 
positivity in this case. From this table, we also see that for $\alpha \ge 0.3$ the 
matrix $\wcalMt \wcalAt^{\alpha} $ has all off-diagonal entries negative, thus 
it is an $M$-matrix and consequently the scheme will preserve positivity.

\begin{table}[ht!]
\caption{ The maximum row-sum (MrowS) and largest off-diagonal entries (MoffD) of matrix  $\wcalMt \wcalAt^{\alpha}$ for the 
one-dimensional problem.}\label{t:positivity}
\centering
\begin{tabular}{|c|cc|cc|cc|cc|}
\hline
\multicolumn{1}{|c|}
 {\multirow{2}{*}{$\alpha$} }  & 
  \multicolumn{2}{| c |}{$ h=1/10$} &  \multicolumn{2}{| c |}{ $ h=1/20$ } &  \multicolumn{2}{| c |}{ $ h=1/40 $} &  \multicolumn{2}{| c |}{ $h=1/80$ }     \\   \cline{2-9}%
          &  MrowS  &  MoffD   &   MrowS     &  MoffD    &   MrowS     &  MoffD   &  MrowS     &  MoffD   \\ 
 \hline
  0.100 & 0.11950 & 0.018190 & 0.05958 & 0.010435  & 0.02977  & 0.005992 &  0.01488 &  0.003442\\
  0.200 & 0.14097 & 0.014288 & 0.07006 & 0.009397  & 0.03498  & 0.006198 &  0.01748 &  0.004089\\
  0.300 & 0.16354 & -0.000425 & 0.08102 & -0.000025 & 0.04042 & -0.000002 & 0.02020 & -0.000001\\
  0.500 & 0.20417 & -0.000852 & 0.10052 & -0.000049 & 0.05006 & -0.000003 & 0.02501 & -0.000002\\
  0.700 & 0.21119 & -0.001177 & 0.10342 & -0.000068 & 0.05142 & -0.000004 & 0.02568 & -0.000003\\ 
  0.800 & 0.18326 & -0.001138 & 0.08963 & -0.000065 & 0.04452 & -0.000004 & 0.02223 & -0.000002\\
  0.900 & 0.11824 & -0.000805 & 0.05786 & -0.000045 & 0.02872 & -0.000003 & 0.01433 & -0.000002\\  
  \hline
\end{tabular}
\end{table}

On Table \ref{t:positivity-f} we report more computations of this kind using refined values 
around $\alpha=0.3$. We see that in the one-dimensional case the matrix becomes an $M$-matrix for $\alpha \ge 0.287$.

\begin{table}[ht!]
\caption{ The maximum row-sum $MrowS$) and largest off-diagonal entries (MoffD) of matrix  $\wcalMt \wcalAt^{\alpha}$
for one-dimensional problem.}\label{t:positivity-f}
\centering
\begin{tabular}{|c|cc|cc|cc|cc|}
\hline
\multicolumn{1}{|c|}
 {\multirow{2}{*}{$\alpha$} }  & 
  \multicolumn{2}{| c |}{$ h=1/10$} &  \multicolumn{2}{| c |}{ $ h=1/20$ } &  \multicolumn{2}{| c |}{ $ h=1/40 $} &  \multicolumn{2}{| c |}{ $h=1/80$ }     \\   \cline{2-9}%
          &  MrowS  &  MoffD   &   MrowS     &  MoffD    &   MrowS     &  MoffD   &  MrowS     &  MoffD   \\ 
 \hline
  0.284 & 0.15991 & 0.000733 & 0.07927  & 0.0004962  & 0.03955  & 0.0003949 &  0.01976 &  0.0002703 \\
  0.286 & 0.16036 &  0.000221 & 0.07949 & 0.0001166  & 0.03966 &  0.0000838 &  0.01982 &  0.0000621 \\
  0.288 & 0.16082 & -0.000302 & 0.07971 & -0.0000436 & 0.03977 & -0.0000015 &  0.01987 & -0.0000001 \\
  0.290 & 0.16127 & -0.000406 & 0.07993 & -0.0000239 & 0.03988 & -0.0000015 &  0.01993 & -0.0000001 \\ 
  0.292 & 0.16172 & -0.000413 & 0.08037 & -0.0000243 & 0.03999 & -0.0000015 &  0.01998 & -0.0000001 \\ 
  \hline
\end{tabular}
\end{table}

We also performed similar computations for the Poisson equation in an $L$-shaped domain,
namely $ \Omega= \{(0,1) \times (0,1)\}\setminus \{(0.5,1) \times (0.5,1) \} $. In this case we
introduce an uniform mesh with step-size in both directions $h=1/(n+1)$ so that the 
stiffness matrix is and $M$-matrix of size $N= 0.75 n^2$. The results are reported in Table \ref{t:positivity-2D}.
We note that the matrix has many negative elements in a row. However, the existence of a 
positive off-diagonal entry in all cases suggests that $\wcalMt \calAt^{\alpha}$ is NOT an $M$-matrix 
when $\wcalMt$ is consistent mass matrix and therefore, the method fails to preserve the positivity 
for all $ \alpha \in (0,1)$. Similar are the results of a rectangular domain on a uniform square mesh. 
These results are a bit different from the one-dimensional computations
shown on Table \ref{t:positivity} and \ref{t:positivity-f}, where we see that for $\alpha \ge 0.3$
we have an $M$-matrix and the method will preserve positivity. We expect that in 3-D problems 
the consistent mass methods will not be positivity preserving.

\begin{table}[ht!]
\caption{ The maximum row sum (MrowS) and the largest off-diagonal entries (MoffD) of matrix  
$\wcalMt \wcalAt^{\alpha}$ for L-shaped domain.}\label{t:positivity-2D}
\centering
\begin{tabular}{|c|cc|cc|cc|}
\hline
\multicolumn{1}{|c|}
 {\multirow{2}{*}{$\alpha$} }  & 
  \multicolumn{2}{| c |}{$ h=1/10$} &  \multicolumn{2}{| c |}{ $ h=1/20$ } &  \multicolumn{2}{| c |}{ $ h=1/40 $} 
  \\   \cline{2-7}%
          &  MrowS  &  MoffD   &   MrowS     &  MoffD    &   MrowS     &  MoffD  \\ 
 \hline
  0.300 & 0.02403 &  0.002695 & 0.00578 &  0.001022 & 0.00143 &  0.000387  \\ 
  0.500 & 0.03820 &  0.004974 & 0.00905 &  0.002483 & 0.00222 &  0.001242 \\ 
  0.700 & 0.04891 &  0.007004 & 0.01170 &  0.004558 & 0.00285 &  0.003007 \\ 
  0.800 & 0.04628 &  0.006979 & 0.01134 &  0.005223 & 0.00275 &  0.003958 \\ 
  0.900 & 0.03176 &  0.004897 & 0.00818 &  0.004222 & 0.00198 &  0.003675 \\ 
  \hline
\end{tabular}
\end{table}

\subsection{Lumped mass finite element method} 
\label{ss:error-lumped}
As discuses in Subsection \ref{ss:fem-lumped}, when we employ mass lumping,  
both the semi-discrete and fully discrete approximations satisfy the positivity property, i.e.,
if $f$ is continuous and $f\ge 0$, then $u_h$ and $w_h$ given by
\eqref{massu} and \eqref{massw} are both non-negative.  This is a
consequence of the fact that the lumped mass matrix is diagonal with
positive diagonal entries, \eqref{bal} and Proposition~\ref{lemma:BURA}.   
Besides, since many finite difference schemes could be considered as obtained by
lumped mass FEM, we have as a by-product of the result below, an 
error estimate for the finite difference approximations of spectral sub-diffusion problems.
We are not aware of rigorous proof of such result.

We conclude this section with an error estimate in the case of two-dimensional problems with 
full regularity and data $f \in H^{1+\gamma}(\Omega)$:

\begin{theorem} \label{masslumped} 
Let $\Omega\subset \RR^2$ and  suppose that \eqref{fullreg} holds.
  Then for $f\in H^{1+\gamma}(\Omega)$ with $\gamma>0$,  $u_h$ given by 
\eqref{massu}  satisfies 
\beq\label{error-masslumped}
\|\calA^{-\alpha}f - u_h \| \le C (h^{2\alpha}+h^{1+\gamma}) \|f\|_{H^{1+\gamma}(\Omega)}
\eeq
with $C$ not depending on $h$.
\end{theorem}

\begin{proof}  Let $\widetilde V_h$ denote the set of continuous piecewise linear
  functions with respect to the mesh on $\Omega$ including non-vanishing
  functions on $\partial \Omega$.   
For the purposes of this proof, we consider $\calI_h$ as a
  map from $C^0(\overline \Omega)$ into $\widetilde V_h$ even though the
  boundary values do not enter into \eqref{massu} (or \eqref{massw}). 
The resulting mass lumped matrix satisfies the  following estimate:
\beq
|(v,w)-(v,w)_h | \le c h^2 \|v\|_a \|w\|_a,\Forall v,w\in \widetilde
V_h
\label{massapp}
\eeq
with $\|\cdot\|_a$ defined by \eqref{a-norm} (same as in the proof of Theorem~\ref{FS-thm}).  

In  addition, the norm $\|\cdot\|=(\cdot,\cdot)_h^{1/2}$
is uniformly equivalent to $\|\cdot\|$ on $\widetilde V_h$ 
 with equivalence constants  independent of $h$.   
We also use well known properties for the interpolant $\calI_h$:
\beq
\|\calI_h f\|+\|\calI_h f\|_a
+h^{-1-\gamma} \|(f-\calI_hf)\|\le C\|f\|_{H^{1+\gamma}(\Omega)}.\label{ihprop}
\eeq

By \eqref{ihprop} and the stability of $\calA^{-\alpha} $,
\beq
\|\calA^{-\alpha} (f-\calI_hf) \le C \|f-\calI_hf\|\le
Ch^{1+\gamma}\|f\|.\label{firsttri}
\eeq
Thus, we are left to bound 
\beq
\| \calA^{-\alpha} \calI_h f -u_h\|\le \|\calA^{-\alpha} \calI_h f-
\calAt^{-\alpha}   \calI_h f \| + \|\calAt^{-\alpha}   \calI_h f-u_h\|.
\label{twoterm}
\eeq
Here $\calAt$ is the finite element operator appearing in \eqref{fea}
and Theorem~\ref{FS-thm}.    By Theorem~\ref{FS-thm}  and \eqref{ihprop},
$$
\|\calA^{-\alpha} \calI_h f-
\calAt^{-\alpha} \calI_h  f\|\le C h^{2\alpha} \| \calI_h f \|\le  C h^{2\alpha}
\|f\|_{H^{1+\gamma}(\Omega).}$$

For the second term in \eqref{twoterm}, we use the Balakrishnan formula and write
\beq
\bal
 \|\calAt^{-\alpha}   \calI_h f-u_h\|    &     =
\frac {\sin(\pi \alpha)} \pi 
\bigg \|\int_0^\infty \mu^{-\alpha} (w_{\mu,h}-\tilde w_{\mu,h}) \, d\mu\bigg \|  \\
          &       \le \frac {\sin(\pi \alpha)} \pi \int_0^\infty\mu^{-\alpha}
\|w_{\mu,h}-\tilde w_{\mu,h}\|\, d\mu
\eal
\label{finalt}
\eeq
where $w_{\mu,h}$ satisfies \eqref{wth} with $f$ replaced by $\calI_h f$
and  $\tilde w_{\mu,h}\in V_h$ satisfies
$$
\mu (\tilde w_{\mu,h},\theta)_h+a(\tilde w_{\mu,h},\theta)=(\calI_h f,\theta)_h
\quad \Forall \theta \in V_h.
$$

It follows that for $e_\mu=w_{\mu,h} -\tilde w_{\mu,h}$ and $\phi \in V_h$, 
\beq
\mu ( e_\mu,\phi)+a(e_\mu,\phi)=\mu \big[(\tilde w_{\mu,h},\phi)_h
-(\tilde w_{\mu,h},\phi)\big]
+\big[(\calI_h f,\phi)-(\calI_h f,\phi)_h \big].
\label{emu}
\eeq
Taking $\phi=e_\mu $ and applying \eqref{massapp} gives
\beq
 \mu \|e_\mu\|^2 + \|e_\mu\|^2_a \le c h^2\big [\mu\|\tilde w_{\mu,h}\|_a
\|e_\mu\|_a + \|\calI_h f\|_a \|e_\mu\|_a\big].
\label{em1}
\eeq
The same argument that showed \eqref{onen} leads to 
$$
\|\tilde w_{\mu,h}\|_a \le \mu^{-1/2} \| \calI_h f\|_h.
$$
Thus, \eqref{em1} implies
$$
 \mu \|e_\mu\|^2 + \|e_\mu\|^2_a \le c h^2[\mu^{1/2} +1] \|f\|_{H^{1+\gamma}(\Omega)}
\|e_\mu\|_a. 
$$
A straightforward application of the arithmetic-geometric mean inequality then
gives
\beq
 \mu \|e_\mu\|^2 + \|e_\mu\|^2_a \le c h^4[\mu +1]
  \|f\|^2_{H^{1+\gamma}(\Omega)}.
\label{emf}
\eeq

We finally bound the integral in \eqref{finalt} by breaking up the
integration interval and bounding the resulting subinterval integrals.  
By \eqref{emf} and the Poincar\'e inequality,
$$\int_0^1\mu^{-\alpha}
\|e_\mu \|\, d\mu\le  c\int_0^1\mu^{-\alpha}
\|e_\mu \|_a\, d\mu \le ch^2
  \|f\|^2_{H^{1+\gamma}(\Omega)}$$
and
$$\int_1^{h^{-2}}\mu^{-\alpha}
\|e_\mu \|\, d\mu\le Ch^2  \|f\|_{H^{1+\gamma}(\Omega)}   \int_1^{h^{-2}}
\mu^{-\alpha} \, d\mu \le C h^{2\alpha}  \|f\|_{H^{1+\gamma}(\Omega)}.
$$
As in \eqref{et2},
$$\|e_\mu\|\le C \mu^{-1} \| \calI_h f \|\le C \mu^{-1} \| f\|
_{H^{1+\gamma}(\Omega)}$$
so that 
$$\int_{h^{-2}}^\infty \mu^{-\alpha}
\|e_\mu \|\, d\mu\le C  \|f\|_{H^{1+\gamma}(\Omega)}\int_{h^{-2}}^\infty \mu^{-1-\alpha}
\, d\mu= C h^{2\alpha}  \|f\|_{H^{1+\gamma}(\Omega)}.$$
Combining the above estimates completes the proof of the
theorem.\end{proof}

\begin{remark} The above theorem and its proof remains valid when
$\Omega\subset \RR^d$ for $d=1$ and $d=3$ provided that $1+\gamma$ is
replaced by $d/2+\gamma$.  We believe that more refined  analysis could lower the 
required smoothness of $f$, but  this will involve solution of number of technical issues that 
are beyond the scope of this paper.
\end{remark}
 
\begin{corollary}\label{total-error}
Let $\Omega\subset \RR^2$ and  suppose that \eqref{fullreg} holds.
Using BURA approximation error estimate \eqref{pbest}, the error bound of the lumped mass finite
element method \eqref{error-masslumped}, and  the inequality $\|f_h\|=\|\pi_h f \| \le \|f\|$ we get 
\beq\label{eq:total-error}
\|\calA^{-\alpha}f - w_h \| \le C (h^{2\alpha}+h^{1+\gamma}) \|f\|_{H^{1+\gamma}(\Omega)} + 
\lambda_1^{-\alpha}E_{\alpha,k}\|f\|
\le C (h^{2\alpha}+h^{1+\gamma} + e^{-2\pi\sqrt{k\alpha}})
\|f\|_{H^{1+\gamma}(\Omega)}
\eeq
with $C$ not depending on $h$ and $k$. Then the contributions of the finite element discretization and the 
BURA approximation to the total error can be balanced choosing properly the parameters $h$ and $k$.
We see that the BURA error is fully controlled by $ E_{\alpha,k}$ and the $L^2$-norm of the data $f$.
This allows to choose $k$ by using Table \ref{t:error} once we fix the desired accuracy of the computations.
To choose the mesh that guarantees the same accuracy 
we need to use either Richardson extrapolation or any other technique for error control 
of the finite element method by mesh refinement. Such discussion is beyond the scope of the paper.
\end{corollary}

\begin{corollary}\label{FD-convergence}
As a by-product of the error analysis of the lumped mass finite element method we also 
obtain an error bound for the finite difference method  for the two-dimensional case trough
it equivalence to the lumped mass approximation on uniform meshes. To the best
of our knowledge, this fact has not been know before.
\end{corollary}

\section{Concluding remarks}\label{section5}
In this paper we study algorithms of optimal complexity for solving the 
system of algebraic equations $\wcalAt^\alpha \tiluh=\tilfh$, 
$0< \alpha <1$ for $\tiluh \in \R^N$, where 
$\wcalAt$ is a symmetric and positive definite $N \times N$ matrix with spectrum 
in $[\lambda_1, \lambda_N]$ which is obtained from discretization of a second order
elliptic problem by a finite difference or finite elements method. Two 
methods, BURA and P-BURA, are analyzed and experimentally studied.
They are based on  the best uniform rational approximation 
$r_{\gamma,k}(t)= P_k(t)/Q_k(t)$ of $t^{\gamma}$. We note that these could be precomputed 
and later used in the computations. Such results for various values of $\alpha$
and $k$ could be found in the report \cite{BURA-Tables-2019}.

The presented estimates show that both methods have exponential 
convergent rate with respect to $k$. They reduce the nonlocal fractional
diffusion problem to solution of small number (determined by $k$) of
systems in the form $(\wcalAt +c \wcalIt) u_h= f_h$, $c \ge 0$.
The algorithm is optimal with respect to $N$, assuming that
solvers of optimal complexity (e.g. multigrid or multilevel) are used   
for the related sparse symmetric and positive definite (discrete elliptic)
problems. More precisely, the computational complexity is $O(kN)$.

The presented numerical tests support the theoretical estimates.
They prove the concept of the new P-BURA method and show its high efficiency.
In contrast to BURA, the accuracy of P-BURA method does not depend on the condition
number of $\calAt$. This makes P-BURA robust with respect to the mesh parameter 
$h$, which also holds true in the case of approximations on locally refined meshes.

In general, the regularity of solution of the considered fractional 
diffusion problems decreases with decreasing of $\alpha$. As shown 
theoretically and numerically in \cite{BP15},
the convergence rate of $||u-u_h||_{L^2}$ is at best  $O(h^{2\alpha})$. 
Here we studied also the lumped mass method that preserves the positivity 
property.  In this context, the used here local mesh refinement for the
CheckerBoard right hand side $f$ shows new promising opportunities 
for a substantial increase of the accuracy based on the robustness 
of P-BURA method. Even more impressive are the obtained results for Example 4
the case of Dirac delta-function right-hand-side.
    
Within the context of this paper, the question about the proper 
norms and algorithms for adaptive mesh refinement is very important, 
but not studied. This holds as well for the case when $f$ has 
lower than $L^2$-regularity. We feel that a  study of these issues needs a 
separate rigorous technical analysis which remains out of the scope of this paper.

\section*{Acknowledgement}
This research has been partially supported by the Bulgarian National Science Fund under grant 
No. BNSF-DN12/1. The work of  R. Lazarov was supported in parts by NSF-DMS \#1620318 grant.

We acknowledge also the provided access to the e-infrastructure of the Centre for 
Advanced Computing and Data Processing, with the financial support by the Grant No.
BG05M2OP001-1.001-0003, financed by the Science and Education for Smart Growth
Operational Program (2014-2020) and co-financed by the European Union through the
European structural and Investment funds.

\bibliographystyle{abbrv}
\bibliography{BURA_refer_exp}  

\end{document}